\documentclass[12pt, oneside]{amsart}
\usepackage{setspace}
\usepackage{verbatim}

\usepackage{amsfonts,fancyhdr,ifthen,bm}
\usepackage{amscd,amssymb} 
\usepackage{times}
\usepackage{graphicx}
\usepackage{color}
\usepackage{epsfig}
\usepackage{mathrsfs}
\usepackage{paralist}
\usepackage{comment}

\usepackage[lmargin=1.35in, rmargin = 1.35in]{geometry}

\DeclareMathOperator{\yy}{\mathbf{y}}
\DeclareMathOperator{\zz}{\mathbf{z}}

\DeclareMathOperator{\vv}{\mathbf{v}}
\DeclareMathOperator{\ww}{\mathbf{w}}

\DeclareMathOperator{\Arows}{A_{\text{Rows}}}

\newtheorem{thm}{Theorem}[section]
\newtheorem*{thm*}{Theorem}
\newtheorem{prop}[thm]{Proposition}
\newtheorem*{prop*}{Proposition}
\newtheorem{cor}[thm]{Corollary}
\newtheorem*{cor*}{Corollary}

\newtheorem{lem}[thm]{Lemma}
\newtheorem*{lem*}{Lemma}

\newtheorem{conj}[thm]{Conjecture}

\newtheorem*{oquest*}{Open Question}

\theoremstyle{remark}
\newtheorem{rmk}[thm]{Remark}

\theoremstyle{remark}
\newtheorem*{rmk*}{Remark}

\theoremstyle{definition}

\theoremstyle{definition}
\newtheorem*{defn*}{Definition}

\theoremstyle{definition}
\newtheorem{ex}[thm]{Example}

\numberwithin{equation}{section}%numbers equations by section

\newcommand{\Z}{\mathbb{Z}}
\newcommand{\QQ}{\mathbb{Q}}

\DeclareMathOperator{\FFF}{\mathbb{F}}

\newcommand{\SelT}{\text{Sel}^{(2)}}

\DeclareFontFamily{U}{wncy}{}
\DeclareFontShape{U}{wncy}{m}{n}{<->wncyr10}{}
\DeclareSymbolFont{mcy}{U}{wncy}{m}{n}
\DeclareMathSymbol{\Sha}{\mathord}{mcy}{"58}

\newcommand{\Lsha}{\mathscr{L}}

\begin{document}
% This is cover.tex

% Insert your information as appropriate.

\title[The density of congruent numbers]{
The congruent numbers have positive natural density
}

\author{
Alexander Smith
}

\thanks{I would like to thank Shou-Wu Zhang for suggesting this problem and for serving as my advisor as I worked on this for my undergraduate thesis. I would also like to thank him and Christopher Skinner for valuable conversations in the course of preparing this paper.}

\date{\today}

\begin{abstract}
We prove that the rational elliptic curve $y^2 = x^3 - n^2x$ satisfies the full Birch and Swinnerton-Dyer conjecture for at least $41.9\%$ of positive squarefree integers $n \equiv 1, 2, 3 \, (8)$ and satisfies the regular BSD conjecture for at least $55.9\%$ of positive squarefree integers $n \equiv 5, 6, 7\,(8)$. In particular, at least $55.9\%$ of positive squarefree integers $n \equiv 5, 6, 7\,(8)$ are congruent numbers. These proofs complete the arguments started by Tian, Yuan, and Zhang in \cite{Tian14}.
\end{abstract}

\maketitle

\section{Introduction}
Suppose $E/\QQ$ is a rank zero elliptic curve of conductor $N$, and take $L(s, E)$ to be its associated $L$ function. One consequence of the full Birch and Swinnerton-Dyer (BSD) conjecture would be that
\[\frac{L(1, E) \cdot (\# E_{tors})^2}{\Omega(E) \prod_{p | 2N} c_p(E)} = |\Sha(E)|\]
where the $c_p$ are Tamagawa factors and $\Omega(E)$ is the least positive real period of $E$. Calling the expression on the left $\Lsha(E)$, this equation would in turn have the following consequence.
\begin{conj}
\label{conj:main}
If $E/\QQ$ is an elliptic curve without rational four torsion, then $\Lsha(E)$ is an odd integer if and only if the $2$-Selmer group $\SelT(E)$ is generated by the image of $E(\QQ)[2] \subseteq E(\QQ)$ in the exact sequence
\[0 \rightarrow E(\QQ)/2E(\QQ) \rightarrow \SelT(E) \rightarrow \Sha(E)[2] \rightarrow 0.\]
\end{conj}
This is a substantially weaker conjecture than full BSD, but has the advantage of seeming tractable. The $2$-Selmer group is readily computable, with Birch and Swinnerton-Dyer giving an effective algorithm to find it in the same article where they first formulated their namesake conjecture \cite{BSD63}. The critical value of the $L$ function is slightly less nice, but can still be readily computed directly or (as was done for this paper) via the Waldspurger formula \cite{Wald85}.

Indeed, for quadratic twists of the congruent number curve, there are already elementary formulations of both sides of this conjecture. For the rest of the paper, take $E$ to be the congruent number curve, and for any positive squarefree integer $n$ take $E^{(n)}$ to be the quadratic twist of $E$ given by Weierstrass equation
\[y^2 = x^3 - n^2x.\]
We will use the notation $E^{(n)}$ to refer to this curve for the rest of the paper.

Write the odd part of $n$ as a product of primes $p_1\dots p_r$. In \cite{Heat94}, Monsky showed that the  $2$-Selmer rank of $E^{(n)}$ can be computed as the corank of a matrix with entries in $\FFF_2$ that is determined entirely by $n$ mod $2$ and the Legendre symbols 
\[\left(\frac{p_i}{p_j}\right),\,\, \left(\frac{-1}{p_i}\right) \mbox{ and  }\left(\frac{2}{p_i}\right) \mbox{ over all } i, j \le r, i \ne j.\]
Much more recently, the parity of $\Lsha(E^{(n)})$ was determined from the same information by Tian, Yuan, and Zhang \cite{Tian14}. We complete their argument to prove the following.
\begin{thm}
\label{thm:rk0}
Conjecture \ref{conj:main} is true for all quadratic twists of the congruent number curve.
\end{thm}
Now, $E^{(n)}$ is a CM elliptic curve with CM field $\QQ(i)$, so we can use Rubin's proof of the main conjecture of Iwasawa theory for imaginary quadratic fields \cite{Rubi91} to say
\[v_p\left(\Lsha(E^{(n)})\big/|\Sha(E^{(n)})|\right) = 0\]
for any $E^{(n)}$ of analytic rank $0$ and any odd prime $p$. Waldspurger's formula implies that $\Lsha(E^{(n)})$ is nonnegative, so we get the following lovely corollary to Theorem \ref{thm:rk0}.
\begin{cor}
The elliptic curve
\[E^{(n)}: y^2 = x^3 - n^2x\]
satisfies the full BSD conjecture if $\SelT(E^{(n)})$ is generated by $2$-torsion, or rather has rank two. In particular, given $t \in \{1, 2, 3\}$, $E^{(n)}$ satisfies full BSD for at least 41.9\% of squarefree positive $n \equiv t \, (8)$.
\end{cor}
The percentage in this result comes from Heath-Brown's computation of the proportion of $E^{(n)}$ with specified $2$-Selmer rank in \cite{Heat94}.

In an apparent non sequitir, Tian, Yuan, and Zhang also use the Gross-Zagier formula to give conditions for when $E^{(n)}$ has analytic rank exactly one, and hence actual rank exactly one. Denote by $\Lsha_x(n)$ the expression in the second column of row $x$ in Table \ref{tab:tyz_tab}. They prove that, if
\begin{itemize}
\item either $\Lsha_{5a}(n)$ or $\Lsha_{5b}(n)$ is nonzero for $n$ equal to $5$ mod 8, or 
\item $\Lsha_{6}(n)$ is nonzero for $n$ equal to $6$ mod $8$, or
\item either $\Lsha_{7a}(n)$ or $\Lsha_{7b}(n)$ is nonzero for $n$ equal to $7$ mod $8$,
\end{itemize}
then $E^{(n)}$ has analytic rank one.

Though the methods of obtaining these conditions are quite different from those for calculating the parity of $\Lsha(E^{(n)})$, the resulting expressions are quite similar. The reason is simple. Similarly to Conjecture \ref{conj:main} we expect that, if $\SelT(E^{(n)})$ has rank three, $E^{(n)}$ should be forced to have analytic rank one. In the case of, for example, $n$ equal to $5$ mod $8$, this condition on Selmer groups is equivalent to a certain singular Monsky matrix $M$ having corank one. But $M$ has corank one if and only if
\[\left| \begin{array}{ll} M & \mathbf{v} \\ \mathbf{v}^{\top} & 0 \end{array} \right|\]
is nonzero for some choice of column vector $\mathbf{v}$. Conditions 5a and 5b amount to finding this determinant for two particular choices of $\mathbf{v}$.

This matrix is defined over $\mathbb{F}_2$, so we typically expect a given $\mathbf{v}$ to have a one half chance of not being in the column space of a corank-one $M$. Justifying this fact formally is hard but possible with an argument coming from Swinnerton-Dyer \cite{Swin08} and Kane \cite{Kane13}. We thus get the following theorem.

\begin{thm}
\label{thm:main}
For $t \in \{5, 6, 7\}$, consider the set
\[ S(t) = \left\{ n \in \mathbb{Z}^{+} : \, n \text{ is squarefree, } n \equiv t \, (8),\text{ and } \text{rk}\left(\SelT(E^{(n)}) \right) = 3 \right\}.\]
For a row $x$ agreeing with $t$, every $n \equiv t \, (8)$ with $\Lsha_x(n) \ne 0$ is in $S(t)$, and the natural density of such $n$ in $S(t)$ is $0.5$. Further, the natural density of $n \in S(5)$ with either $\Lsha_{5a}(n)$ or $\Lsha_{5b}(n)$ nonzero is $0.75$, and the natural density of $n \in S(7)$ with either $\Lsha_{7a}(n)$ or $\Lsha_{7b}(n)$ nonzero is $0.75$.
\end{thm}

The natural density of $n \equiv t \, (8)$ with $ \text{rk}\left(\SelT(E^{(n)}) \right) = 3$ is slightly greater than $.8388$, so this theorem has the following nice consequence.
\begin{thm}
\label{thm:boom}
\hfill 
\begin{itemize}
\item Of the positive squarefree integers equal to $5$ mod $8$, at least $62.9$\% are congruent numbers.
\item Of the positive squarefree integers equal to $6$ mod $8$, at least $41.9$\% are congruent numbers.
\item Of the positive squarefree integers equal to $7$ mod $8$, at least $62.9$\% are congruent numbers.
\end{itemize}
\end{thm}

This paper blackboxes the results of Tian, Yuan, and Zhang, but the success of these methods for quadratic twists of the congruent number curve makes a potential generalization of their work appealing. Since $2$-Selmer groups of curves with full two torsion are particularly simple objects, our hope is that their results can be extended to this larger family. We have made some progress on this question in the analytic rank zero case that we hope to give in a subsequent paper, but the analytic rank one case remains elusive. There is more work to be done in this area. However, for this paper, we content ourselves with the one quadratic twist family.
\section{Reduction to linear algebra}
Throughout this section, $n$ will be a positive, squarefree integer, and $E^{(n)}$ will denote the corresponding twist of the congruent number curve. Our first goal is defining the expressions $\Lsha_x(n)$ that appear in the second column and the matrices that appear in the third column of Table \ref{tab:tyz_tab}.

To begin, we define an additive version of the Legendre symbol by
\[\left(\frac{d}{p} \right)_+ := \frac{1}{2}\left(1 - \left( \frac{d}{p} \right)\right)\]
We will always take this to be an element of $\mathbb{F}_2$.

Write the odd part of $n$ as a product $p_1p_2 \dots p_r$ of odd primes. We then define a column vector $\yy = (y_1, \dots, y_r)$ with values in $\mathbb{F}_2$ by
\[y_i :=\left( \frac{-1}{p_i} \right)_+\]
and another column vector $\zz = (z_1, \dots, z_d)$ by
\[z_i := \left( \frac{2}{p_i} \right)_+\]
We also define a $r \times r$ matrix $A$ over $\mathbb{F}_2$ by
\[A_{ij} = \begin{cases} \left( \frac{p_j}{p_i} \right)_+ &\text{ for } i \ne j \\
\sum_{j \ne i} A_{ij} &\text{ for } i = j.\end{cases}\]
For $\mathbf{v}$ any column vector, take $D_{v}$ to be the diagonal matrix so that $(D_v)_{ii} = \mathbf{v}_i$ for all $1 \le i \le d$.

Next, take $g(n) \in \mathbb{F}_2$ to be the order of $2 \text{Cl}(\QQ(\sqrt{-n}))$ mod $2$. It is nonzero if and only if $\QQ(\sqrt{-n})$ has an element in its class group of exact order four. There is an elementary method for calculating $g(n)$ first given by R\'{e}dei \cite{ReRe33} that we summarize here. Via Gauss genus theory, $\text{Cl}(\QQ(\sqrt{-n}))[2]$ consists of those classes corresponding to id\`{e}les that are trivial everywhere except at the ramified primes of the field, or rather the primes dividing the discriminant $D$  of $\QQ(\sqrt{-n})$. Among these consider one of squarefree ideal norm $d | D$; it will be nonprincipal if $d$ is neither $1$ nor $n$. But, by calculating norms, this ideal is the square of another in the class group if and only if
\[x^2  + ny^2 = dz^2\]
is soluble in $\QQ$ for some squarefree $d | D$ other than $d = 1, n$. Via Hasse's principle, this is soluble if the Hilbert symbols $(d, -n)_p$ equal $+1$ at all places $p$. This is equivalent to saying that $d > 0$ and that $(d, -n)_p = +1$ at all primes dividing $D$. But $(a, -n)_p(b, -n)_p = (ab, -n)_p$. If we retake the Hilbert symbol as a map to $\mathbb{F}_2$, so it maps to $0$ and $1$ instead of $+1$ and $-1$, this equation is equivalent to linearity. Then finding such a $d$ is equivalent to solving a set of linear equations, and we can express $g(n)$ as a determinant. We do so in Table \ref{tab:red}, with the notation coming from Section \ref{ssec:first}.
\begin{center}
\begin{table}
\begin{tabular}{c | l}
$n$ mod $4$ & $g(n)$ \\
\hline & \\[-7pt]
1 & $\text{det} \bigg(
				A\big[[r], [r] -\{i\}\big] \,\,\,\,\, \zz  \bigg)$ \\[9pt]
	& \,\,\,\,\, for any $1 \le i \le r$ \\[7pt]
2 & $\text{det} \left( A + D_z \right)$ \\[7pt]
3 & $\text{det} \, A\big[[r] - \{i\}, [r] - \{j\}\big] $ \\[5pt]
    & \,\,\,\,\,  for any $1 \le i, j \le r$
\end{tabular}
\caption{R\'{e}dei matrices for four torsion of class groups of $\QQ{\sqrt{-n}}$ with $n > 0$.}\label{tab:red}
\end{table}
\end{center}

Using $g(n)$, we can define $\Lsha(n)$ and subsequently $\Lsha_x$ for each row $x$ of Table \ref{tab:tyz_tab}.

\begin{defn*}
Take $\Lsha: \mathbb{Z} \rightarrow \mathbb{F}_2$ to be the recursive function defined by
\begin{itemize}
\item $\Lsha(1) = 1$
\item If $n \, \equiv \, 1 \, (8)$ is positive and squarefree, and if $p$ is any prime divisor of $n$,
\begin{equation}
\label{eq:n1rec}
\Lsha(n) = \sum_{\substack{p | d | n \\ d\, \equiv\, 1 \, (8)}} g(d)\Lsha(n/d).
\end{equation}
\item Otherwise, $\Lsha(n) = 0$.
\end{itemize}
\end{defn*}

Now that we understand what the $\Lsha_x$ are, we can precisely formulate the results from Tian, Yuan, and Zhang.
\begin{thm} \label{thm:TYZ} \cite[Theorem 1.1, 1.2]{Tian14}
Take $E$ to be the congruent number curve and assume $n$ is a positive squarefree integer. If $n\, \equiv\, x \,(8)$ for $x \in \{1, 2, 3\}$, we have
\[\Lsha(E^{(n)}) \equiv \Lsha_x(n) \mod 2.\]
Further, if $n$ is $5$, $6$, or $7$ mod $8$, and if the row $x$ agrees with $n$ mod $8$, then the analytic rank of $E^{(n)}$ is exactly one if $\Lsha_x(n)$ is nonzero.
\end{thm}

The main contribution of this paper is a reinterpretation of these formulae as determinants of matrices.
\begin{thm}
\label{thm:tab_power}
With $A$, $\yy$, and $\zz$ defined from $n$ as above, the second and third columns of Table \ref{tab:tyz_tab} are equal in the rows corresponding to $n$ mod $8$.
\end{thm}

\begin{center}
\begin{table}
\label{tab:tyz_tab}
\begin{tabular}{c | l | l}
$n$ mod 8 & Tian-Yuan-Zhang expression $\Lsha_x$ & Determinant reformulation $\det M_x$\\
\hline & &\\[-6pt]
$1$ & $\Lsha(n)$ &\,\, $\left|\begin{array}{ll}
				A + A^{\top} & A^{\top} \\
				A                  & D_{z} \end{array}\right|$ \\[14pt]
$2$ &
		$\sum\limits_{\substack{d | n \\ d\, \equiv\, n \, (16)}} g(d)\Lsha(n/d)$
			     & \,\,$\left|\begin{array}{ll}
				A + A^{\top} + D_{y + z} & A^{\top} \\
				A                  & D_{z} \end{array}\right|$ \\
$3$ &
		$\sum\limits_{\substack{d | n \\ d\, \equiv\, 3 \, (8)}} g(d)\Lsha(n/d)$
			     & \,\,$\left|\begin{array}{lll}
				A + A^{\top}  & A^{\top}  & \yy \\
				A                  & D_{z}       & 0\\
				\yy^{\top}      & 0              & 0\end{array}\right|$ \\[25pt]
$5$ (a) &
		$\sum\limits_{\substack{d | n \\ d\, \equiv\, 5 \, (8)}} g(d)\Lsha(n/d)$
			     & \,\,$\left|\begin{array}{lll}
				A + A^{\top}  & A^{\top}  & \yy + \zz \\
				A                  & D_{z}       & 0\\
				\yy^{\top} + \zz^{\top}      & 0              & 0\end{array}\right|$ \\[25pt]
\,\,\, (b) &
		$\sum\limits_{\substack{d_0d_1 | n \\ d_0\, \equiv\, 7 \, (8) \\ d_1\, \equiv\, 3 \, (8)}} g(d_0)g(d_1)\Lsha(n\big/d_0d_1)$
			     & \,\,$\left|\begin{array}{lll}
				A + A^{\top}  & A^{\top}  & 0 \\
				A                  & D_{z}      & \yy\\
				0                  & \yy^{\top}& 0          \end{array}\right|$ \\[40pt]
$6$ &
		$\begin{array}{l} \sum\limits_{\substack{d_0d_1 | n \\ d_0\, \equiv\, 7n \, (16) \\ d_1 \, \equiv \, 7 \, (8)}} g(d_0)g(d_1)\Lsha(n\big/d_0d_1)  \\
						+\sum\limits_{\substack{d | n \\ d\, \equiv\, n \, (16)}} g(d)\Lsha(n/d) \end{array}$ 
			     & \,\,$\left|\begin{array}{lll}
				A + A^{\top}+ D_{y + z}  & A^{\top}  & \yy  \\
				A                  & D_{z}       & \yy\\
				\yy^{\top}  & \yy^{\top}      & 0 \end{array}\right|$\\[35pt]

$7$ (a) &
		$\sum\limits_{\substack{d | n \\ d\, \equiv\, 7 \, (8)}} g(d)\Lsha(n/d)$
			     & \,\,$\left|\begin{array}{llll}
				A + A^{\top}  & A^{\top}  & \yy + \zz & 0 \\
				A                  & D_{z}       & 0 & \yy \\
				\yy^{\top} + \zz^{\top}      & 0  &0            & 0 \\
				0 & \yy^{\top} &0 &0 \end{array}\right|$ \\[40pt]
\,\,\, (b) &
		$\sum\limits_{\substack{d_0d_1 | n \\ d_0\, \equiv\, 5 \, (8) \\ d_1 \, \equiv \, 3\,(8)}} g(d_0)g(d_1)\Lsha(n\big/d_0d_1)$
			     & \,\,$\left|\begin{array}{llll}
				A + A^{\top}  & A^{\top}  & \yy + \zz & \yy \\
				A                  & D_{z}       & 0 & 0 \\
				\yy^{\top} + \zz^{\top}      & 0  &0            & 0 \\
				\yy^{\top} & 0 &0 &0 \end{array}\right|$ \\[25pt] 
\end{tabular}
\vspace*{6mm}
\caption{Matrix forms for formulae from \cite{Tian14}.}
\end{table}
\end{center}

We will refer to the matrix in the third column row $x$ of this table by $M_x$. Given Theorem \ref{thm:tab_power}, it is straightforward to prove Theorem \ref{thm:rk0}.
\begin{proof}[Proof of Theorem \ref{thm:rk0}]

Per Monsky's calculations in \cite{Heat94},
\begin{itemize}
\item $\text{rk}\left(\SelT(E^{(n)})\right) = 2 + \text{crnk}(M_1)$ for $n \equiv 1 \, (8)$,
\item $\text{rk}\left(\SelT(E^{(n)})\right) = 2 + \text{crnk}(M_2)$ for $n \equiv 2 \, (8)$, and
\item $\text{rk}\left(\SelT(E^{(n)})\right) = 1 + \text{crnk}(M_3\big[[2r], [2r]\big])$ for $n \equiv 3 \, (8)$, with the bracket notation defined below.
\end{itemize}

With Theorem \ref{thm:tab_power} and Theorem \ref{thm:TYZ}, we immediately get the theorem for $n \equiv 1, 2 \, (8)$. In the final case, we see that the column space of $M_3[[2r], [2r]]$ is contained in the set of $\mathbf{v} = (v_1, \dots, v_{2r})$ with $\sum_{i=1}^r v_i = 0$. $\mathbf{y}$ is not in this space, and similarly $\mathbf{y}^{\top}$ is not in the corresponding row space, so $M_3[[2r], [2r]]$ (see notation below) has rank two less than $M_3$. This is enough to give the result.
\end{proof}
The rest of this section will be spent proving Theorem \ref{thm:tab_power}. Subsequently, in Section \ref{sec:density}, we will push through the necessary analysis of the matrices $M_x$ to prove Theorem \ref{thm:main}. Theorem \ref{thm:tab_power} splits into eight cases, but the analysis of the first case will take as many words as the remaining seven cases combined. We turn to it now.

\subsection{The first case of Theorem \ref{thm:tab_power}}
\label{ssec:first}
We need some technical definitions.

\begin{defn*}
If $m$ is a positive integer, take $[m]$ to be the set of integers $\{1, 2, \dots, m\}$.
\end{defn*}
\begin{defn*}
Suppose $B = (B_{ij})$ is a $m \times m$ matrix over $\mathbb{F}_2$. Let $S = (s_1, \dots, s_{m'})$ and $C = (c_1, \dots, c_{m''})$ be two subsets of $[m]$, with $s_i < s_j$ and $c_i < c_j$ for $i < j$. We define $B[S, C]$ to be the $m' \times m''$ matrix defined by
\[(B[S, C])_{ij} = B_{s_ic_j}.\]
If $m' = m''$, we denote by $B_{\text{Rows}}[S, C]$ the variant of this normalized to have rows summing to zero, namely
\[(B_{\text{Rows}}[S, C])_{ij}  = \begin{cases} B_{s_ic_j} &\mbox{if } i \ne j \\  \sum_{k \ne i} B_{s_ic_k} &\mbox{if } i = j.\end{cases}\]
If $\mathbf{v} = (v_1, \dots, v_m)$ is a $m$ dimensional column vector, $\mathbf{v}[S]$ is defined as $(v_{s_1}, \dots, v_{s_{m'}})$.
\end{defn*}

\begin{defn*}
Given a square matrix $M$ defined as a function $M(B, \mathbf{v}_1, \dots, \mathbf{v}_t)$ with $B$ an $m \times m$ matrix with rows summing to zero and with the $\mathbf{v}_i$ all $m$ dimensional vectors, and given a subset $S \subseteq [m]$, define
\[M(B, \mathbf{v}_1, \dots, \mathbf{v}_t)[S] := M(\Arows[S, S], \mathbf{v}_1[S], \dots, \mathbf{v}_t[S]).\] 

\end{defn*}
\begin{rmk}
Think of $M_x$ as a function $M_x(A, \zz)$, with $A$ corresponding to the odd part $p_1 p_2 \dots p_r$ of $n$. If $d = p_{i_1} \dots p_{i_{r'}}$ is a positive integer divisor of the odd part of $n$,  we have that the $M_x$ that corresponds to $d$ is $M_x(A, \zz)[S]$, where $S = \{i_1, \dots, i_{r'}\}$. This is the rationale for this definition.
\end{rmk}

The $0 \times 0$ matrix is nonsingular, so to show that $\Lsha(n) = \det M_1$ for all $n$, we just need to show that \eqref{eq:n1rec} is satisfied with $\det M_1$ replacing $\Lsha$. The proper way to formulate this is as follows.
\begin{prop}
\label{prop:main_rec}
Take $\yy$ and $\zz$ to be any $r$ dimensional column vectors over $\mathbb{F}_2$, and take $A$ to be an $r \times r$ matrix over $\mathbb{F}_2$ whose rows sum to zero and which satisfies
\[A_{ij} + A_{ji} = y_iy_j\]
for $i$ and $j$ not equal.

Consider $M_1$ as a function of $A$ and $\zz$. Define a new function
\[Q(A, \zz) = \bigg(A\big[[r], [r-1]\big] \,\,\,\,\, \zz  \bigg).\]
Then we have
\begin{equation}
\label{eq:main_rec}
 \det M_1(A, \zz)
\end{equation}
\[ = \sum_{\substack{S \subseteq [r]\\ 1 \in S}} \left(1 + \sum_{i \in S} y_i \right)\left(1 + \sum_{i \in S} z_i\right) \det Q(A, \zz)[S] \,\cdot\, \det M_1(A, \zz)[S'].\]
In this expression, $S'$ denotes the complement of $S$.
\end{prop}
The rest of Section \ref{ssec:first} will cover the proof of this proposition.

Write $F(A, \zz)$ for the right hand side of \eqref{eq:main_rec}, and take $M = M_1$. Notice that, if we fix $A$, we can think of $F$ and $\det M$ as polynomials in $\zz$. With this mindset, we will prove the lemma by proving $F(A, \zz)$ and $\det M(A,  \zz)$ have the same $P_k(\zz) = z_{i_1} \dots z_{i_k}$ component for any set of indices $(i_1, \dots, i_k)$. We do this inductively on $k$. Notice that, by permuting the matrix $A$ and the vectors $\yy$ and $\zz$, we can without loss of generality assume that $i_j = j$ for each $j$. Further, notice that $F(A, (z_1, \dots, z_r))$ and $F(A, (z_1, \dots, z_k, 0, \dots, 0))$ have the same $z_1\dots z_k$ component, and similarly for $\det M(A, (z_1, \dots, z_r))$ and $\det M(A, (z_1, \dots, z_k, 0, \dots, 0))$.

We first prove that the constant and $z_1$ coefficients are equal to zero for $F$ and $\det M$ if $r \ge 1$. Take $z_i = 0$ for $i > 1$; we wish to prove $F$ and $\det M$ are both zero. For $M$, we first note that, as $M$ is symmetric, we can calculate
\[\det M = \sum_{\sigma} \prod_i (M)_{i \sigma(i)} = \sum_{\sigma, \sigma^2 = 1} \prod_i M_{i \sigma(i)}\]
where the sum is over all permutations of $[2d]$. This result comes from pairing the $\sigma$ term with the $\sigma^{-1}$ whenever the two are distinct.

But this involution $\sigma$ must fix an even number of elements from $[2d]$. That is, $\sigma$ hits an even number of diagonal entries of $M$. In particular, the expression of the determinant as a polynomial in the $z_i$ only has terms of even degree. As $A$ is singular, $\det M = 0$ when all the $z_i$ are $0$, and since the determinant polynomial only has terms of even degree, $\det M$ must also be zero if $z_i = 0$ for $i > 1$.

At the same time, if $z_i = 0$ for $i > 1$, we get
\[F(A, \zz) = (1 + z_1)\left(1 + \sum_{i} y_i\right) \det Q(A, \zz).\]
But $\det Q(A, \zz)$ is divisible by $z_1$, so this expression is divisible by $z_1(1 + z_1) = 0$.

We next prove the result for $k = 2$, so take $z_i = 0$ for $i > 2$. If $\sum_i y_i = 1$, then both sides of \eqref{eq:main_rec} are zero, so assume $\sum_i y_i = 0$. Choose any $j \le r$. We then have
\[ F(A, \yy, \zz) = (1 + z_1 + z_2)\det Q(A, \zz)\]
\[ = (1 + z_1 + z_2) \left(z_1 \det A\big[[r] - \{1\}, [r] - \{j\}\big] + z_2 \det A\big[[r] - \{2\}, [r] - \{j\}\big]\right).\]
Then the coefficient of $z_1z_2$ in $F(A, \zz)$ is 
\[\det A\big[[r] - \{1\}, [r] - \{j\}\big] + \det A\big[[r] - \{2\}, [r] - \{j\}\big].\]
The sum of the rows of $A$ is $\yy^{\top}$. Thus, using multilinearity of determinant, we see that this coefficient is equal to
\[ \det \left( \begin{array}{c} \left(\yy\big[[r]-\{j\}\big]\right)^{\top} \\ A\big[[r] - \{1, 2\}, [r] - \{j\}\big] \end{array} \right) = \sum_{i \ne j} y_i \det A\big[[r] - \{1, 2\}, [r] - \{i, j\}\big].\]
If $y_i = 1$, there are an odd number of nonzero $y_j$ with $j \ne i$, so we can unfix $j$ to get that the coefficient is
\[\sum_j \sum_{i \ne j} y_iy_j\det A\big[[r] - \{1, 2\}, [r] - \{i, j\}\big]\]

On the other side, we have
\[\det M =  \sum_{\sigma, \sigma^2 = 1} \prod_i M_{i \sigma(i)}.\]
For a permutation $\sigma$ to give a nonzero $z_1z_2$ term, we must have $\sigma(r+1) = r+1$ and $\sigma(r+2) = r+2$, with $\sigma(r+k) \le r$ for all $k > 2$. Then there are precisely two distinct integers $i, j \le r$ so $\sigma(i) \le r$ and $\sigma(j) \le r$. To avoid the zeros on the diagonal, then, we must have $\sigma(i) = j$. For $M_{ij}$ to be nonzero, we also need $y_iy_j = 1$. We then can write
\[\det M = z_1z_2 \sum_j \sum_{i \ne j} y_iy_j\det A\big[[r] - \{1, 2\}, [r] - \{i, j\}\big],\]
which agrees with $F$.

At this point, we have established the proposition for terms with $k \le 2$. We now prove it for a general monomial of length $k$. Thus assume $\zz = (z_1, \dots, z_k, 0, \dots, 0)$. Let $S$ be a subset of $[r]$ containing $1$. If $S$ contains more than two elements of $[k]$, and if $S'$ denotes the complement of $S$ in $[r]$, then we know that 
\[\left(1 + \sum_{i \in S} y_i \right) \left(1 + \sum_{i \in S} z_i \right) \det Q(A, \zz)[S] \det M(A, \zz)[S']\]
cannot have a nontrivial $z_1\dots z_k$ component, as $(1 + \sum_{i \in S} z_i)\det F$ is quadratic in the variables from $\zz[S]$, so any monomial component of it can reference at most $2$ elements of $\zz[S]$. We can thus restrict the sum to being over $S$ with $|S \cap [k]| = 2$, at which point we can use our proof of the proposition for $k \le 2$ to rewrite $\det Q(A, \zz)[S]$ in terms of $M(A, \zz)[S]$. Then $F(A, \zz)$ has the same $z_1 \dots z_k$ term as
\begin{equation}
\label{eq:FfromM}
\sum_{\substack{1 \in S \subseteq [d] \\ |S \cap [k]| = 2}} \det M(A, \zz)[S] \det M(A, \zz)[S'].
\end{equation}

Then, to finish the proof of Proposition \ref{prop:main_rec}, we need to prove the following proposition.

\begin{prop}
\label{prop:F_rec}
The $z_1 \dots z_k$ term of $\det M(A,  \zz)$ equals the $z_1 \dots z_k$ term of \eqref{eq:FfromM}.
\end{prop}

\begin{proof}
We have
\[\det M = \sum_{\sigma, \sigma^2 = 1} \prod_i M_{i \sigma(i)}.\]
For a $\sigma$ to contribute to the $z_1\dots z_k$ term, we must have $\sigma(r+j) = r+j$ for $1 \le j \le k$. It may have $\sigma(i) = i$ for no other $i$, or the contribution will be zero. Let $\Omega$ be the set of involutions of $[2r]$ with these restrictions.Then we can write the coefficient for the $z_1 \dots z_k$ term as
\[\sum_{\sigma \in \Omega} \prod_{\substack{i \le r \\ \sigma(i) \le r}}  (y_iy_{\sigma(i)})^2 \prod_{\substack{i > r + k \\ i - r \ne \sigma(i)}} a_{i-r, \sigma(i)}^2 \prod_{\substack{i > r+k \\ i - r= \sigma(i)}} \left( \sum_{\substack{j \le r \\ j \ne i -r}} a_{i-r, j} \right)^2\]
Given a $\sigma \in \Omega$, let $\tau$ be a function that has domain the set of $j \le r$ so $\sigma(j + r) = j$ and has codomain $[r]$, subject to the restriction that $\tau(j) \ne j$ for each $j$. Then, using $(\sigma, \tau) \in \Omega'$ to denote a pair of this form, we can write the above coefficient as
\begin{equation}
\label{eq:sig_tau}
\sum_{(\sigma, \tau) \in \Omega'} \prod_{\substack{i \le r \\ \sigma(i) \le r}} y_iy_{\sigma(i)} \prod_{\substack{i > r + k \\ i - r \ne \sigma(i)}} a_{i-r, \sigma(i)} \prod_{\substack{i > r+k \\ i - r= \sigma(i)}} a_{i-r, \tau(i - r)}
\end{equation}
Given $(\sigma, \tau) \in \Omega'$, let $\mathscr{B} = \mathscr{B}(\sigma, \tau)$ denote the minimal set of indices that contains $r+1$ and so that

\begin{enumerate}
\item For $i \le r$, $i \in \mathscr{B}$ if and only if $i + r \in \mathscr{B}$.
\item $i \in \mathscr{B}$ if and only if $\sigma(i) \in \mathscr{B}$.
\item If $\sigma(i + r) = i$ and $\tau(i) \in \mathscr{B}$, then $i \in \mathscr{B}$.
\end{enumerate}

We claim that $\mathscr{B}$ will contain precisely one other $r+ l \le r + k$ besides $r + 1$. To see this, consider the graph with vertices lying in $[2r]$, an edge from $i$ to $i+r$ for all $i \le r$, and an edge from $i$ to $\sigma(i)$ for all $i \le 2r$. As $\sigma(r+l) = r+l$ for $l \le k$, $r + l$ has degree one in this graph for $1 \le l \le k$. Every other vertex will have degree two. Then the connected component containing $r+1$ in this graph must be a path, and the terminus of this path must be some other $r+l$ where $1 \ne l \le k$. Then the closure of $1$ under the first two rules will contain exactly one other $r+l$ where $1 \ne l \le k$. Since $\sigma(i + r) \ne i$ for $i \le k$, the third rule does not add any such $r + i \le r+k$.

But then we have a bijection between $(\sigma, \tau) \in \Omega'$ and tuples $(S, \sigma_1, \tau_1, \sigma_2, \tau_2)$ so that
\begin{enumerate}
\item $S$ is contained in $[r]$ and $S \cap [k] = \{1, l\}$ for some $l$.
\item $\sigma_1$ is an involution of  $\mathscr{B}_S = \{i \in [2r] : i \in S \text{ or } i- r \in S\}$ that fixes $\{r+1, r+l\}$ and no other elements.
\item $\sigma_2$ is an involution of the complement $\mathscr{B}_{S'}$ of $\mathscr{B}_S$ that fixes $\{r+j : 2 \le j \le k\text{ and } j \ne l\}$ and no other elements. (Here, $S'$ denotes the complement of $S$ as always).
\item $\tau_1$ is a function from $\{i \in S : \sigma(i + r) = i\}$ to $S$ so that $\tau_1(i) \ne i$ for all $i$ the domain.
\item $\tau_2$ is a function from $\{i \in S': \sigma(i + r) = i\}$ to $S'$ so that $\tau_2(i) \ne i$ for all $i$ in the domain.
\item The closure of $1$ with respect to $(\sigma_1, \tau_1)$ is the entirety of $\mathscr{B}_S$.
\end{enumerate}
For an $S$ obeying the first condition, we denote by $\Omega'[S']$ the set of pairs $(\sigma_2, \tau_2)$ obeying the third and fifth condition. In analogy to \eqref{eq:sig_tau}, we have that $\det M(A, \zz)[S']$ has the same $z_2 \dots z_{l-1}z_{l+1} \dots z_k$ coefficient as
\[\sum_{(\sigma_2, \tau_2) \in \Omega'[S']} \prod_{\substack{i \in S' \\ \sigma_2(i) \le r}} y_iy_{\sigma_2(i)} \prod_{\substack{i - r \in S' \\i > r + k \\ i - r \ne \sigma_2(i)}} a_{i-r, \sigma_2(i)} \prod_{\substack{i - r \in S' \\ i - r = \sigma_2(i)}} a_{i-r , \tau_2(i - r)}.\]
Note that the sixth condition adds an extra complication to our analysis of $(\sigma_1, \tau_1)$, but this will turn out not to matter. For $1 \in S$, take $\Omega_0[S]$ to be the set of elements $(\sigma_1, \tau_1)$ of $\Omega'[S]$ such that the closure of $1$ with respect to the pair is $\mathscr{B}_S$. The bijection above establishes that the $z_1 \dots z_k$ coefficient of $\det M(A, \zz)$ equals

\[\sum_{\substack{1 \in S \subseteq [r] \\ |S \cap [k]| = 2}} \left( \sum_{(\sigma, \tau) \in \Omega'[S']} \,\,\,\,\ \prod_{\substack{i  \in S' \\ \sigma(i) \le r}} y_iy_{\sigma(i)} \,\,\,\, \prod_{\substack{i - r \in S' \\i > r + k \\ i - r \ne \sigma(i)}} a_{i-r, \sigma(i)} \,\,\,\, \prod_{\substack{i - r \in S' \\ i - r = \sigma(i)}} a_{i-r, \tau(i - r)}\right)\]
\[\cdot \left( \sum_{(\sigma, \tau) \in \Omega_0[S]} \,\,\,\,\ \prod_{\substack{i  \in S \\ \sigma(i) \le r}} y_iy_{\sigma(i)} \,\,\,\, \prod_{\substack{i - r \in S \\i > r + k \\ i - r \ne \sigma(i)}} a_{i-r, \sigma(i)} \,\,\,\, \prod_{\substack{i - r \in S  \\ i - r = \sigma(i)}} a_{i-r, \tau(i - r)}\right).\]
But this equals the $z_1 \dots z_k$ coefficient of

\begin{equation}
\label{eq:almost_rec}
\sum_{\substack{1 \in S \subseteq [d] \\ |S \cap [k]| = 2}}  \det M(A, \zz)[S']
\end{equation}
\[\cdot z_1z_l\left( \sum_{(\sigma, \tau) \in \Omega_0[S]} \,\,\,\,\ \prod_{\substack{i  \in S \\ \sigma(i) \le r}} y_iy_{\sigma(i)} \,\,\,\, \prod_{\substack{i - r \in S \\i > r + k \\ i - r \ne \sigma(i)}} a_{i-r, \sigma(i)} \,\,\,\, \prod_{\substack{i - r \in S  \\ i - r = \sigma(i)}} a_{i-r, \tau(i - r)}\right).\]
Unlike the sum over $\Omega'[S']$, we cannot immediately simplify the sum over $\Omega_0[S]$ to a determinant. However, denoting by $\Omega_1$ the set of $(\sigma, \tau) \in \Omega'[S]$ that are not in $\Omega_0[S]$, we claim that, for any $S$,
\[\left( \sum_{(\sigma, \tau) \in \Omega_1[S]} \,\,\,\,\ \prod_{\substack{i  \in S \\ \sigma(i) \le r}} y_iy_{\sigma(i)} \,\,\,\, \prod_{\substack{i - r \in S \\i > r+ k \\ i - r \ne \sigma(i)}} a_{i-r, \sigma(i)} \,\,\,\, \prod_{\substack{i - r \in S \\ i - r = \sigma(i)}} a_{i-r, \tau(i - r)}\right) = 0.\]

For, by \eqref{eq:almost_rec}, this is equal to the sum over nonempty $S_1 \in \{S_1 \subset S : S_1 \cap [k] = \emptyset\}$ of the product of
\begin{equation}
\label{eq:mas1}
\det M(A, \zz)[S_1]
\end{equation}
with another factor, where $S_1$ in this case denotes the indices outside the closure of $1$. But \eqref{eq:mas1} is zero if $S_1 \ne \emptyset$ as $S_1$, by definition, has no index of a nonzero $\zz$ component. Then adding the permutations in $\Omega_1(S)$ makes no difference, and we get that the $z_1 \dots z_k$ component of $\det M$ is equal to
\[\sum_{\substack{1 \in S \subseteq D \\ |S \cap [k]| = 2}} \det M(A, \zz)[S] \cdot \det M(A, \zz)[S']\]
as claimed.
\end{proof}
This finishes the proof of Proposition \ref{prop:main_rec} for $t=1$, proving the first case of Theorem \ref{thm:tab_power}.

\subsection{The seven other cases of Theorem \ref{thm:tab_power}} The recurrences we want again come from enumerating over $(\sigma, \tau)$ pairs and using closures. We start by handling the two easiest cases, 3 and 5a, showing the strength of this strategy. For $\mathbf{u}$ an arbitrary $r$ dimensional vector, consider
\[\left| \begin{array}{lll} A + A^{\top} & A^{\top} & \mathbf{u} \\ A & D_z & 0 \\ \mathbf{u}^{\top} & 0 & 0 \end{array} \right|.\]
Suppose we want to find the $z_1 \dots z_k$ term of this. Given a contributing $(\sigma, \tau)$ pair, the closure of $2r + 1$ will contain exactly one $z_l$ with $1 \le l \le k$. Then the $z_1 \dots z_k$ term of this determinant equals that of
\[\sum_{\substack{S \subset [d] \\ |S \cap [k] | = 1}} \left| \begin{array}{lll} 0 & \Arows[S, S]^{\top} & \mathbf{u}[S] \\ \Arows[S, S] & D_{z[S]} & 0 \\ \mathbf{u}[S]^{\top} & 0 & 0 \end{array} \right| \cdot \det M_1(A, \zz)[S']\]
But the first determinant here is linear in the $z_i$, so the restriction on $S$ is unnecessary. Then the sum does not depend on $k$, so we get an expression that simply equals our original determinant. Rewriting the first determinant in this sum in the more compact form $\det O(A, \zz, \mathbf{u})$ where we define
\[O(A, \mathbf{v}_1, \mathbf{v}_2) :=  \left( \begin{array}{ll} A & \mathbf{v}_1 \\ \mathbf{v}_2^{\top} & 0 \end{array} \right), \]
we get
\[\left| \begin{array}{lll} A + A^{\top} & A^{\top} & \mathbf{u} \\ A & D_z & 0 \\ \mathbf{u}^{\top} & 0 & 0 \end{array} \right| = \sum_{S \subseteq [r]} \det O(A, \zz, \mathbf{u})[S] \cdot \det M_1(A, \zz)[S'] .\]
Taking $\mathbf{u} = \mathbf{y}$ in this equation, we see that, for the summand indexed by $S$ to be nonzero, we need $\sum_{i \in S} y_i \ne 0$, since otherwise the first $r$ columns of the first matrix sum to zero. But then we need $\sum_{i \in S} z_i \ne 0$, since otherwise the first $r$ rows will sum to zero. If both these are satisfied, so that the $d$ corresponding to $S$ is equal to $3$ mod $8$, we get that this determinant equals $g(d)$ per Table \ref{tab:red}. This gives row 3.

Taking $\mathbf{u} = \mathbf{y} + \mathbf{z}$, we see that, to contribute to the sum, the $n'$ corresponding to $S$ is either $7$ or $5$ mod $8$. But if it is $7$ mod $8$, then the determinant would also be divisible by $\sum_{i \in S} z_i$, which is zero. So $d$ must be $5$ mod $8$, and we see that the first determinant is again $g(d)$. This finishes the argument for row 5a.

Row 7b is not much harder. If $\mathbf{v}$ is another column vector, we get
\[\left| \begin{array}{llll} A + A^{\top} & A^{\top} & \mathbf{u} & \mathbf{v} \\ A & D_z & 0 & 0 \\ \mathbf{u}^{\top} & 0 & 0 & 0 \\ \mathbf{v}^{\top} & 0 & 0 & 0\end{array} \right| \]
\[= \sum_{\substack{S \subseteq [r] \\ S_0 \subseteq S}} \det O(A, \zz, \mathbf{u})[S_0] \cdot \det O(A, \zz, \mathbf{v})[S - S_0] \cdot \det M_1(A, \zz)[S'] .\]
Taking $\mathbf{u} = \mathbf{y}$ and $\mathbf{v} = \mathbf{y} + \mathbf{z}$ and again referring to Table \ref{tab:red} then gives us row 7b.

Next up is 7a. Take $\mathbf{w}$ to be another column vector, and consider
\begin{equation}
\label{eq:7a_gen}
\left| \begin{array}{llll} A + A^{\top} & A^{\top} & \mathbf{u} & 0 \\ A & D_z & 0 & \mathbf{w} \\ \mathbf{u}^{\top} & 0 & 0 & 0 \\ 0 & \mathbf{w}^{\top} & 0 & 0\end{array} \right|.
\end{equation}
In a contributing $(\sigma, \tau)$ pair, the closure of row $2r + 2$ will either contain $2r + 1$ or exactly one index between $r+1$ and $2r$. In the latter case, to go through the standard argument of proving two formulae the same by proving that they have the same $z_1 \dots z_k$ coefficient, we will need a linearization of
\[\left| \begin{array}{lll} A + A^{\top} & A^{\top} & 0 \\ A & D_z & \mathbf{w} \\ 0 & \mathbf{w} & 0 \end{array} \right|\]
with respect to the $z_i$. Such a linearization is given by $\det N$ with $N$ defined as
\[N(A, \mathbf{z}, \mathbf{w}) := \left( \begin{array}{llll} A + A^{\top} & A^{\top} & 0 & 0 \\ A & 0 & \mathbf{z} & \mathbf{w} \\  0 & \mathbf{z}^{\top} & 0 & 0 \\ 0 & \mathbf{w}^{\top} & 0 & 0\end{array}\right).\]
Then \eqref{eq:7a_gen} is
\begin{equation}
\label{eq:7a_develop}
\sum_{S \subseteq [r]} \det O(A, \mathbf{w}, \mathbf{u})[S] \cdot \det M_1(A, \zz)[S']
\end{equation}
\[ + \sum_{\substack{S \subseteq [r] \\ S_0 \subseteq S}} \det O(A, \zz, \mathbf{u})[S_0] \cdot \det N(A, \zz, \mathbf{w})[S - S_0] \cdot \det M_1(A, \zz)[S']\]
where we have split the sum into those $(\sigma, \tau)$ where $2r+2$ has closure containing $2r+1$ and those where it does not. 

Take $\mathbf{w} = \yy$ and $\mathbf{u} = \mathbf{y} + \mathbf{z}$. Note that $\det N(A, \zz, \yy)$ is zero if $n$ is $3$ mod $4$, since the first $r$ columns of $N$ then sum to $0$. Additionally, if $n$ is $1$ mod $8$, then the first $2r$ columns of $N(A, \zz, \yy)$ sum to zero. So $\det N$ is nonzero only if $n$ is $5$ mod $8$. But then the second sum in \eqref{eq:7a_develop} is zero if $n$ equals $7$ mod $8$, as $S_0$ must correspond to $5$ mod $8$, $S- S_0$ to $5$ mod $8$, and $S'$ to $1$ mod $8$ for the summand to be zero, and
\[5 \cdot 5 \cdot 1 \not\equiv 7 \mod (8).\]
Note that, if $S$ corresponds to a divisor $d$ equaling $7$ mod $8$, then $\det O(A, \mathbf{y}, \mathbf{y} + \mathbf{z})[S]$ is equal to $g(d)$. This finishes case 7a.

For case 5b we just need to prove that, if $n \equiv 5 \,(8)$, then
\begin{equation}
\label{eq:5b_is_easy}
\det N(A, \zz, \yy) = \sum_{\substack{d | n \\ d\, \equiv \, 3\, (8)}} g(d) g(n/d)
\end{equation}
The closure of column $2r + 1$ in a contributing $(\sigma, \tau)$ pair will certainly contain column $2r + 2$, and it will do so with a single pass through $A + A^{\top}$. If we define $S$ as the subset of $[r]$ in the closure of $2r + 1$ up to the pass through $A + A^{\top}$, we can use the normal argument to prove
\[\det N(A, \zz, \yy) = \sum_{S \subseteq [r]} \det O(A, \zz, \yy)[S] \det O(A, \yy, \yy)[S'].\]
Take $d | n$ corresponding  to $S$. Then $d \equiv 3 \, (8)$ for any contributing $S$. This forces $n/d$ to be $7$ mod $8$, in which case the second term is $g(n/d)$. We thus get \eqref{eq:5b_is_easy}, finishing case 5b.

We just have cases 2 and 6 left, and we start by considering
\[P(A, \yy, \zz) = \left( \begin{array}{ll} D_{y + z}  & A^{\top} \\ A & D_z \end{array} \right).\]
We claim that the determinant of this equals $\det(A + D_z)$ if $\sum_i y_i = 0$ and is otherwise $0$. Suppose first that $\sum_i y_i = 1$. Then we see that the $2r$ dimensional vector $(\yy + \zz, \zz)$ is the sum of the columns of $P$. Then the corank of
\[\left( \begin{array}{ll} (\yy + \zz) \cdot (\yy + \zz)^{\top}  + D_{y + z}  & (\yy + \zz) \cdot \zz^{\top} + A^{\top} \\  \zz \cdot (\yy + \zz)^{\top} + A & \zz \cdot \zz^{\top} + D_z \end{array} \right)\]
is exactly one greater than that of $P$ (we see that $(\yy + \zz, \zz)$ is not in this column space since the sum of its elements is not $0$). But this is an alternating matrix that hence has even rank. Then $P$ has odd rank, and $\det P = 0$.

So assume $\sum_i y_i = 0$. With elementary row and column operations, we find
\[\det P(A, \yy, \zz) = \left| \begin{array}{ll} \yy \cdot \yy^{\top} & A^{\top} + D_z \\ A  + D_z & D_z \end{array} \right|\]
If $\yy = 0$, then the determinant of this matrix is $\det (A + D_z)$, as claimed. Otherwise, permute $A$ so $y_i = 1$ for $1 \le i \le l$ for some $l$ and $y_i = 0$ for $i > l$. Then define an $r \times r$ matrix $C$ by
\[C_{ij} = \begin{cases} (A + D_z)_{ij} - (A + D_z)_{i1} & \text{for } 1 < j \le l \\
(A + D_z)_{ij} & \text{otherwise.}\end{cases}\]
$C$ is derived from $A$ by adding the first column to columns two through $l$. In particular, take $D_{(1, 0, \dots0)}$ to be the $r \times r$ diagonal matrix that is nonzero only in the upper left hand corner. Then we get that $\det P(A, \zz, \yy)$ equals
\[\left| \begin{array}{ll} D_{(1, 0, \dots0)} & C^{\top} \\ C & D_z \end{array} \right| = \det C + \det \left( \begin{array}{ll} 0 & C\big[[r], [r] - \{1\}\big] ^{\top} \\ C\big[[r], [r] - \{1\}\big]  & D_z \end{array} \right).\]
Notice that the sum of the $r-1$ columns of $C\big[[r], [r] - \{1\}\big]$ is $\zz$. Then the second determinant here equals
\[\det \left( \begin{array}{ll} 0 & C\big[[r], [r] - \{1\}\big] ^{\top} \\ C\big[[r], [r] - \{1\}\big]  & \zz \cdot \zz^{\top} + D_z \end{array} \right).\]
But this is an alternating matrix of odd dimension, so it has odd corank, so it has determinant zero. Then $\det P(A, \zz, \yy) = \det C = \det (A + D_z)$ when $\sum_i y_i = 0$.

Now look at $M_2$. Given a $(\sigma, \tau)$ pair, we see that the set of $i \le r$ with $\sigma(i) = i$ and the set of $i \le r$ with $\sigma(i) \le r$ but $i \ne \sigma(i)$ have different closures. We thus get
\[\det M_2(A, \yy, \zz) = \sum_{S \subseteq [r]} \det P(A, \yy, \zz)[S] \cdot \det M_1(A, \zz)[S'].\]
But our result above shows that, if $d | n$ corresponds to $S$ and $d \equiv 1 \, (4)$, then $\det P(A, \yy, \zz)[S] = g(2d)$. This gives case 2.

Consider case 6. We split $\det M_6$ as
\[\left|\begin{array}{lll}
				A + A^{\top}+ D_{y + z}  & A^{\top}  & 0 \\
				A                  & D_{z}       & \yy\\
				0 & \yy^{\top}      & 0 \end{array}\right| \,\,+\,\, \left|\begin{array}{lll}
				A + A^{\top}+ D_{y + z}  & A^{\top}  & \yy  \\
				A                  & D_{z}       & 0\\
				\yy^{\top}  & 0      & 0 \end{array}\right|.\]
Given a $(\sigma, \tau)$ pair, the closure of $2r + 1$ in the first matrix either contains one term from the diagonal $D_{y + z}$ or one term from the diagonal $D_z$. This determinant is then
\[\sum_{S \subseteq[r]} \left(\det O(A, \yy, \yy + \zz)[S] + \det N(A, \zz, \yy)[S]\right) \cdot \det M_2(A, \yy, \zz)[S'].\]
The $N$ term here cannot contribute, as $2 \cdot 1 \not\equiv 6 \, (8)$. The $O$ term is potentially nonzero when $S$ corresponds to $d \equiv 7\,(8)$ or $d \equiv 5 \, (8)$, but the latter case cannot contribute to the sum since it forces $n/d \equiv 6 \, (8)$. Then this determinant is
\[\sum_{ \substack{d | n \\ d\, \equiv \, 7 \, (8)}} g(d) \Lsha_2(n/d),\]
which is the second part of the sum.

Now consider the second determinant in our split of $\det M_6$. Write
\[T(A, \yy, \zz) = \left(\begin{array}{lll}
				 D_{y + z}  & A^{\top}  & \yy \\
				A                  & D_{z}       &  0\\
				 \yy^{\top}   &  0  & 0 \end{array}\right).\]

Considering the closure of $2r + 1$ together with the closure of the diagonal elements, we get
\[ \left|\begin{array}{lll}
				A + A^{\top}+ D_{y + z}  & A^{\top}  & \yy  \\
				A                  & D_{z}       & 0\\
				\yy^{\top}  & 0      & 0 \end{array}\right|
				= \sum_{S \subseteq [r]} \det T(A, \yy, \zz)[S] \cdot \det M_1(A, \yy, \zz)[S'].\]

With row and column additions, we find
\[\det T = \left|\begin{array}{lll}
				 0 & A^{\top} + D_z  & \yy \\
				A + D_z      & D_{z}       &  0\\
				 \yy^{\top}   &  0  & 0 \end{array}\right|
		= \left| \begin{array}{ll} A+ D_z  & \zz \\ \yy^{\top} & 0 \end{array} \right|. \]
Notice that $\yy^{\top}$, having elements summing to one, is not in the space spanned in the other rows. Then this determinant is nonzero if and only if the first $r$ rows of this $(r + 1) \times (r + 1)$ matrix have rank $r$. But notice $\zz$ is in the column space of $A + D_z$, so the first $r$ rows have this rank if and only if $\det(A + D_z) \ne 0$. We thus have
 \[\left|\begin{array}{lll}
				A + A^{\top}+ D_{y + z}  & A^{\top}  & \yy  \\
				A                  & D_{z}       & 0\\
				\yy^{\top}  & 0      & 0 \end{array}\right|
				= \sum_{S \subseteq [r]} \det \left(\Arows[S, S] + D_{z[S]}\right) \cdot \det M_1(A, \yy, \zz)[S'].\]
But this gives the first part of the sum in case 6, so $\det M_6 = \Lsha_6$, as claimed. This was the final case, so Theorem \ref{thm:tab_power} is proved.

\section{Positive density results}
\label{sec:density}
We now wish to find how often the $\Lsha_x$ are nonzero. This problem is very similar to finding the Selmer rank of a quadratic twist with full two torsion, but it is enough removed that we cannot just use the main results of Swinnerton-Dyer in \cite{Swin08} and Kane in \cite{Kane13}. Instead, we need to use the worst kind of generalization of these papers, giving what amounts to the wordy, technical underpinning of \cite[Theorem 3]{Kane13} and \cite[Theorem 1]{Swin08} in our Theorem \ref{thm:ksd_ugly}. That said, this also serves as an opportunity to clarify these two papers, whose arguments were rendered rather abstract by the lack of a concrete object to consider.

Recall that we defined $\left( \frac{d}{p} \right)_+$ to be the element in $\mathbb{F}_2$ given by
\[\left( \frac{d}{p} \right)_+ = \frac{1}{2} \left(1 - \left(\frac{d}{p} \right) \right).\]
We set up Theorem \ref{thm:ksd_ugly} with the following definitions.
\begin{defn*}
Take $D$ to be the product of a set of distinct odd primes. Choose $n_0$ to be an odd integer coprime to $D$, and take $X_{n_0, D}$ to be the set of squarefree integers $n$ coprime to $2D$ such that $n/n_0$ is positive and a quadratic residue mod $8D$.

Take $T_1$, $T_2$, $Q_1$, $Q_2$ to be finite sequences of integer divisors of $2D$, with $T_1$ and $T_2$ of the same length, and write $T_l$ as $\left(d_{l,1}, \dots, d_{l,t}\right)$. Finally, take $d_{diag}$ to be a divisor of $2D$, and take $C$ to be an arbitrary alternating $t \times t$ matrix over $\mathbb{F}_2$. Taking $b_l$ to be the product of the elements in $Q_l$, we assume that none of $b_1$, $b_2$, $-b_1b_2$ is a square, and that one of $b_1, b_2$ is not $-1$ times a square.

Given $n = p_1 \cdot \dots \cdot p_r$, we now define a $2r + t \times 2r + t$ matrix piece by piece. For $l = 1, 2$, define an alternating $r \times r$ matrix by
\[(B_l)_{ij} = \begin{cases} \sum_{d \in Q_l}  \left( \frac{d}{p_i}\right)_+ \left( \frac{d}{p_j}\right)_+ &\text{if } i \ne j \\ 0 &\text{if } i = j.\end{cases}\]
Also define $t \times r$ matrices $R_l$ by
\[(R_l)_{ij} = \left(\frac{d_{l,i}}{p_j}\right)_+.\]

Finally, take $A$ as in previous section, and take $D_{diag}$ to be the diagonal matrix defined by
\[(D_{diag})_{ii} = \left(\frac{d_{diag}}{p_i}\right)_+.\]
Then define
\[M^{alt} = \left( \begin{array}{lll} B_1 & A^{\top} + D_{diag} & R_1^{\top} \\ A + D_{diag} & B_2 & R_2^{\top} \\ R_1 & R_2 & B \end{array} \right). \]

$M^{alt}$ is a function of tuples of primes $(p_1, \dots, p_r)$ whose product $n$ is in $X_{n_0, D}$. We typically ignore the dependence on the ordering of this tuple, writing this function as $M^{alt}(n)$.
\end{defn*}
\begin{defn*}(Corank correction)
Let $D$, $n_0$, $B$, $T_1$, $T_2$, $Q_1$, $Q_2$, and $d_{diag}$ be as in the above definition, so that $M^{alt}(n)$ is defined for $n = p_1\cdot \dots \cdot p_r \in X_{n_0, D}$. Take $\vv$ to be the sum of the first $r$ columns of $M^{alt}$, take $\vv'$ to be the sum of columns $r+1$ to $2r$, and take $\vv_i$ to be column $2r + i$ for $1 \le i \le t$.

Then define $\delta$ to be the least integer so that $\{\vv, \vv', \vv_1, \dots, \vv_t\}$ has rank $t + 2 - \delta$ for some choice of $n \in X_{n_0, D}$.
\end{defn*}

\begin{thm}
\label{thm:ksd_ugly}
Let $D$, $n_0$, $B$, $T_1$, $T_2$, $Q_1$, $Q_2$, and $d_{diag}$ be as in the above definitions, so that $M^{alt}(n)$ is defined for $n = p_1\cdot \dots \cdot p_r \in X_{n_0, D}$. Given this info, take $\delta$ as in the previous definition.

For $k \ge 0$, define
\[\alpha_k = 2^{k+1} \cdot \prod_{j=1}^k (2^j - 1)^{-1} \cdot \prod_{j=0}^{\infty} (1 + 2^{-j})^{-1}.\]
Then
\[\lim_{N \rightarrow \infty} \frac{\big| \left\{ n \in X_{n_0, D} \,\,: \,\, |n| \le N \, \text{ and } \,\text{corank}\left(M^{alt}(n)\right) = k + \delta \right\} \big|} {\big| \left\{ n \in X_{n_0, D} \,\,: \,\, |n| \le N  \right\}\big|} \]
\[ = \begin{cases} \alpha_k &\text{if } k \ge 0 \text{ and } k \equiv t +\delta \mod 2 \\ 0 & \text{otherwise.}\end{cases}\]
\end{thm}

\begin{rmk*}
The technical conditions on $b_1$ and $b_2$ are necessary, and interesting things happen if we ignore them. See Example \ref{ex:AAT}.
\end{rmk*}
We first use this result to prove Theorem \ref{thm:main}. We will need the following simple proposition.

\begin{prop}
\label{prop:sel3}
Given $n$, and given  $x \in \{5a, 5b, 6, 7a, 7b\}$ agreeing with $n$ mod $8$, $\Lsha_x(n)$ can only be nonzero if $\SelT(E^{(n)})$ has rank exactly three.
\end{prop}
\begin{proof}
If $n \equiv 5 \, (8)$, then the minimal possible corank of $M_1$ is one. It achieves this corank if and only if $\SelT(E^{(n)})$ has rank three per Monsky. Since $M_{5a}$ and $M_{5b}$ are constructed from $M_1$ by adding one row and one column, neither can have corank more than one less than that of $M_1$.

Similarly, if $n \equiv 7 \, (8)$, then the minimal possible corank of $M_1$ is two. It achieves this corank if and only if $\SelT(E^{(n)})$ has rank three per Monsky. Since $M_{7a}$ and $M_{7b}$ are constructed from $M_1$ by adding two rows and two columns, neither can have corank more than two less than that of $M_1$.

Finally, if $n \equiv 6 \, (8)$, then the minimal possible corank of $M_2$ is one. It achieves this corank if and only if $\SelT(E^{(n)})$ has rank three per Monsky. Since $M_6$ is constructed from $M_2$ by adding one row and one column, it cannot have corank more than one less than that of $M_2$.
\end{proof}

To apply Theorem \ref{thm:ksd_ugly}, we need to convert the matrices of the last section to alternating matrices. We do this now.
\begin{prop}
\label{prop:to_ksd}
For every function $F$ in the left column of Table \ref{tab:tyz_alt}, use the data in the remaining columns, in addition to $d_{diag} = 1$ and $B = 0$, to define a set $X_{n_0, 8}$ and a matrix $M^{alt}(n)$ for $n \in X_{n_0, 8}$. Then, for $n \in X_{n_0, 8}$, $F(n)$ is nonzero if and only if $\text{crnk}(M^{alt}(n)) = \delta$, where $\delta$ is as defined above and is given for reference as the final column of the table.
\end{prop}

\begin{center}
\begin{table}
\begin{tabular}{l | c | c | c | c | c | c | c}
Function $F$ & \,\, $n_0$ \,\, & \,\,$t$ \,\,& $T_1$ & $T_2$ & $Q_1$ & $Q_2$ & \,\, $\delta$\,\, \\
\hline
$\Lsha_{5a}(n)$& 5 & 1 & (-2) & (2) & (-1) & (2) &  1 \\
$\Lsha_{5b}(n)$& 5 &1 & (1) & (-1) & (-1) & (2) &  1 \\
$\Lsha_{5a}(n) + \Lsha_{5b}(n)$& 5 & 1 & (-2) & (-2) & (-1) & (2) &  1 \\
$\Lsha_6(2n)$ & 3 or 7 & 1 & (2) & (-2) & (-2, -1) & (2) & 1 \\
$\Lsha_{7a}(n)$ & 7  & 2 & (-2, 1) & (1, -2) & (-1) & (2) & 0  \\
$\Lsha_{7b}(n)$ & 7 & 2 & (-2, -1) & (1, 1) & (-1) & (2) & 0 \\
$\Lsha_{7a}(n) + \Lsha_{7b}(n)$ & 7 & 2 & (-2, -1) & (1, -2) & (-1) & (2) & 0
\end{tabular}\\[10pt]
\caption{Reformulating $\Lsha_x$ so Theorem \ref{thm:ksd_ugly} can be applied}\label{tab:tyz_alt}
\end{table}
\end{center}

\begin{proof}
We start by looking at $\Lsha_{5a} + \Lsha_{5b}$, which per Table \ref{tab:tyz_tab} is given as the determinant of the matrix
\[\left( \begin{array}{lll}A + A^{\top} & A^{\top} & \yy + \zz \\ A & D_z & \yy \\ \yy^{\top} + \zz^{\top} & \yy^{\top} & 0 \end{array} \right)\]
We see that $(0, \zz, 1)$ is in the column space of this matrix as the sum of the first $2r$ columns (where $0$ is a $r$ dimensional vector and $1$ is a scalar). Then this determinant is nonzero if and only if
\[\left( \begin{array}{llll}A + A^{\top} & A^{\top} & \yy + \zz & 0 \\ A & D_z & \yy  & \zz \\ \yy^{\top} + \zz^{\top} & \yy^{\top} & 0 & 1 \end{array} \right)\]
has rank $2r + 1$. But the rank of this equals that of
\[\left( \begin{array}{llll}A + A^{\top} & A^{\top} & \yy + \zz & 0 \\ A & D_z + \zz \cdot \zz^{\top} & \yy + \zz  & \zz \\ \yy^{\top} + \zz^{\top} & \yy^{\top} + \zz^{\top} & 1 & 1 \end{array} \right)\]
We note that $(0, \zz, 1)$ is a linearly independent column, as all the others have their first $2r$ entries summing to $0$. Then this has rank $2r + 1$ if and only if
\[\left( \begin{array}{lll}A + A^{\top} & A^{\top} & \yy + \zz \\ A & D_z + \zz \cdot \zz^{\top} & \yy + \zz \\ \yy^{\top} + \zz^{\top} & \yy^{\top} + \zz^{\top} & 1 \end{array} \right)\]
has corank one. But notice that the $2r \times 2r$ submatrix in the upper left hand corner has corank at least two, as it is alternating and has rows summing to zero. Then this last matrix has corank one if and only if the $2r \times 2r$ left corner submatrix has corank two and $(\yy + \zz, \yy + \zz)$ is outside the column space of this corner submatrix. But this is the case if and only if 
\[\left( \begin{array}{lll}A + A^{\top} & A^{\top} & \yy + \zz \\ A & D_z + \zz \cdot \zz^{\top} & \yy + \zz \\ \yy^{\top} + \zz^{\top} & \yy^{\top} + \zz^{\top} & 0 \end{array} \right)\]
has corank one. This is the form of $M^{alt}$ given by the data in the third row of Table \ref{tab:tyz_alt}. That $\delta = 1$ for this $M^{alt}$ comes from the fact that $\vv' + \vv$ is zero while $\vv = (\yy, 0, 1)$ and $\vv_1 = (\yy+ \zz, \yy + \zz, 0)$ are typically independent.

The same argument works for $\Lsha_{5a}$ and $\Lsha_{5b}$, so we have dealt with $n \equiv 5\, (8)$.

Next consider $\Lsha_{7a} + \Lsha_{7b}$, which per Table \ref{tab:tyz_tab} is given as the determinant of the matrix
\[\left( \begin{array}{llll}A + A^{\top} & A^{\top} & \yy + \zz & \yy \\ A & D_z & 0 & \yy \\ \yy^{\top} + \zz^{\top} & 0 & 0 & 0 \\ \yy^{\top} & \yy^{\top} & 0 & 0 \end{array} \right).\]
Noting that $(0, \zz, 0, 1)$ is the sum of columns $r+1$ to $2r$, we get that this matrix has the same rank as that of
\[\left( \begin{array}{lllll}A + A^{\top} & A^{\top} & \yy + \zz & \yy  & 0\\ A & D_z + \zz \cdot \zz^{\top} & 0 & \yy +\zz & \zz \\ \yy^{\top} + \zz^{\top} & 0 & 0 & 0 & 0\\ \yy^{\top} & \yy^{\top} +\zz^{\top} & 0 & 1 & 1 \end{array} \right).\]
But the last column is still the sum of columns $r+1$ to $2r$, so the rank of this equals the rank of
\[\left( \begin{array}{lllll}A + A^{\top} & A^{\top} & \yy + \zz & \yy  \\ A & D_z + \zz \cdot \zz^{\top} & 0 & \yy +\zz  \\ \yy^{\top} + \zz^{\top} & 0 & 0 & 0 \\ \yy^{\top} & \yy^{\top} +\zz^{\top} & 0 & 1  \end{array} \right).\]
So $\Lsha_{7a} + \Lsha_{7b}$ equals the determinant of this matrix. The $1$ in the lower right corner does not contribute to the determinant, as the $2r+1 \times 2r + 1$ upper left minor of this matrix is alternating and hence of positive corank. Removing this $1$, we get the $M^{alt}$ given by the seventh row of Table \ref{tab:tyz_alt}. The same argument works for $\Lsha_{7a}$ and $\Lsha_{7b}$. In these cases, we see that $\vv$, $\vv'$, $\vv_1$, and $\vv_2$ are typically independent, so $\delta = 0$ for this $M^{alt}$

Next up, we see that $\Lsha_6$ is nonzero if and only if
\[ \left( \begin{array}{lllll} A + A^{\top} + D_{y+z}  & A^{\top} & \yy & \yy + \zz & 0 \\
					A & D_z & \yy & 0 & \zz \\
					\yy^{\top} & \yy^{\top} & 0  & 1 & 1 \end{array}  \right)\]
has rank $2r + 1$. This has the same rank as
\[\left( \hspace{-7pt}\begin{array}{lllll}\begin{array}{l}(\yy + \zz) \cdot (\yy + \zz)^{\top} \\ \,\,\,\,\,\,\,\,\,\, +  \yy \cdot \yy^{\top} + D_{ z} \end{array}  & \begin{array}{l}A^{\top}  \\ \,\end{array} & \begin{array}{l}\zz  \\ \,\end{array} & \begin{array}{l}\yy + \zz  \\ \,\end{array} & \begin{array}{l}0  \\ \,\end{array}\\

				\begin{array}{l}A\end{array}     & D_z + \zz \cdot \zz^{\top} & \yy + \zz & 0 & \zz \\
				\begin{array}{l}\zz^{\top}\end{array} & \yy^{\top} + \zz^{\top} & 0 & 1 & 1
					 \end{array} \right).\]
If $\sum_i z_i = 0$, the sum of the first $2r$ columns is $(0, \zz, 1)$, and $(\yy + \zz, 0, 1)$ is linearly independent from the the other columns by virtue of having first $r$ elements not summing to zero. If $\sum_i z_i = 1$, we get that $(0, \zz, 1)$ is independent and $(\yy + \zz, 0, 1)$ is dependent. Either way, we see that $\Lsha_6$ is nonzero if and only if the $M^{alt}$ given in Table \ref{tab:tyz_alt} has corank exactly one, which we can calculate to be the $\delta$ for $M^{alt}$. This finishes the proof.
\end{proof}

With Proposition \ref{prop:to_ksd}, Proposition \ref{prop:sel3}, and Theorem \ref{thm:ksd_ugly}, we can easily prove Theorem \ref{thm:main}.
\begin{proof}[Proof of Theorem \ref{thm:main}]
Choose $n_0 \in \{5, 6, 7\}$. We know that, among the positive squarefree integers $n$ equaling $n_0$ mod $8$, the proportion with $\SelT(E^{(n)})$ of rank three is $\alpha_1$, where $\alpha_k$ is as defined in Theorem \ref{thm:ksd_ugly}. At the same time, if $x$ agrees with $n_0$, then the proportion of the squarefree integers equaling $n_0$ mod $8$ that have $\Lsha_x$ nonzero is $\alpha_0$ per Proposition \ref{prop:to_ksd} and Theorem \ref{thm:ksd_ugly}. Since $\alpha_0/\alpha_1 = 0.5$, and since $\Lsha_x(n)$ is nonzero only if $\SelT(E^{(n)})$ is of rank three, we get that $\Lsha_x$ is nonzero for half of the $n$ with $\SelT(E^{(n)})$ of rank three

With $\Lsha_{5a} + \Lsha_{5b}$, we also see that exactly one of $\Lsha_{5a}(n)$ and $\Lsha_{5b}(n)$ is nonzero for half of the $n \equiv 5\, (8)$ with $\SelT(E^{(n)})$ of rank three. We then can calculate that the chance that either $\Lsha_{5a}(n)$ or $\Lsha_{5b}(n)$ is nonzero, given that $\SelT(E^{(n)})$ is of rank three, is $0.75$. A similar analysis works for $n \equiv 7\, (8)$, and this gives the theorem.
\end{proof}

At this point, the only thing we need to do is to prove Theorem \ref{thm:ksd_ugly}. This will take the next two subsections.

\subsection{Theorem \ref{thm:ksd_ugly} for bit assignments}
For this section, we presume that we have $D$, $n_0$, $B$, $T_1$, $T_2$, $Q_1$, $Q_2$, and $d_{diag}$ as in Theorem \ref{thm:ksd_ugly}, so that $M^{alt}(n)$ is defined for $n = p_1 \cdot \dots \cdot p_r \in X_{n_0, D}$. Take $c$ to equal the number of prime divisors of $D$. Then $M^{alt}$ is determined by $(c + 2)(r-1) + \frac{1}{2}r(r-1)$ total bits. The first part of this sum comes from the Legendre symbols
\[\left(\frac{-1}{p_i} \right),\,\, \left(\frac{2}{p_i} \right),\,\,\left(\frac{q}{p_i} \right)\]
for $1 \le i \le r$ and $q$ any prime factor of $D$, subject to the conditions on $n$ mod $8D$. The second part comes from 
\[\left( \frac{p_i}{p_j} \right)\]
for all $i, j$ with $1 \le i < j \le r$. Then a more algebraic version of Theorem \ref{thm:ksd_ugly} is the following.
\begin{prop}
\label{prop:swin_main}
Suppose we are in the context of Theorem \ref{thm:ksd_ugly}, so $M^{alt}$ is determined by $ \frac{1}{2}(r + 2c + 4)(r-1)$ bits. Then
\[\lim_{r \rightarrow \infty} \frac{\text{Number of bit assignments with } \text{crnk}(M^{alt}) = k + \delta}{2^{\frac{1}{2}(r + 2c + 4)(r-1)}}\]
\[ = \begin{cases} \alpha_k &\text{if } k \ge 0 \text{ and } k \equiv t + \delta \mod 2 \\ 0 & \text{otherwise.}\end{cases}\]
\end{prop}
\begin{proof}
Consider $\vv$, the sum of the first $r$ columns of $M^{alt}$. The construction of $B_1$ and restrictions on $n$ mod $8D$ force there to be a $d | 2D$ so row $i$ of this vector equals $\left(\frac{d}{p_i}\right)_+$ for $i \le r$. For the next $r$ elements, this column clearly equals $\left(\frac{d_{diag}}{p_i}\right)_+$, and the choice of $n_0$ forces the final $t$ elements of the vector to be constant. Then, by expanding $t$ by one, we can choose $T_1$, $T_2$, and $B$ so $M^{alt}$ has a copy of $\vv$ among the last $t+1$ columns. We can do the same for $\vv'$, the sum of columns $r+1$ to $2r$, again incrementing $t$. Call $R_1'$, $R_2'$, and $B'$ the resulting matrices from this expansion, with $t' = t+ 2$. Then the corank of $M^{alt}$ is the same as the corank of
\[M^{alt}_S= \left(\begin{array}{lll} B_1[S,S] & (A+ D_{diag})[S,S]^{\top} & R_1'\big[[t'], S\big]^{\top} \\
				(A + D_{diag})[S, S] & B_2[S, S] & R_2'\big[[t'], S\big]^{\top} \\
				R_1'\big[[t'], S\big] & R_2'\big[[t'], S\big] & B'  \end{array}\right)\] 
for $S = [r-1]$. 

We first assume that $\delta = 0$. If $\delta$ is positive, it corresponds to a set $S_t \subseteq [t]$ so that the product of the corresponding divisors in $T_1'$ is a square, the product of the corresponding divisors in $T_2'$ is a square, and the sum of the corresponding rows of $B'$ is zero. Then we can just remove a row and column corresponding to $i \in S_t$, decreasing both $\delta$ and corank by one and thus not effecting the veracity of the proposition. So we can assume $\delta = 0$. 

We also assume that there is no such $S_t$ so that the product of the corresponding divisors in $T_l$ is a square for both $l = 1$ and $l = 2$. Suppose otherwise. Via row and column operations, we can assume that it is true for $S_t = \{1\}$. Then $\vv_1$ is nonzero somewhere in its last $t' - 1$ rows and is zero in its first $2r + 1$ rows. Say it is nonzero at row $2r + 2$. We can then add $\vv_1$ to various other columns, and add $\vv_1^{\top}$ to various other rows, to make $R_1'$ and $R_2'$ zero in the second column and make $B'$ zero in the second column and second row except at $(B')_{1,2}$ and $(B')_{2,1}$. We can then remove both these rows and their corresponding columns without changing either the corank of our matrix or $\delta$. Then the assumption does not change the veracity of the proposition.

Our approach now is quite similar to Swinnerton-Dyer. We see that the distribution of $M^{alt}_{[s]}$ only depends on $s$, and not $r$, so long as $r > s$. Then to find the limit of the distribution of coranks of $M^{alt}_{[r-1]}$ as $r$ heads towards infinity, we just find the distribution of coranks of $M^{alt}_{[s]}$ as $s$ heads to infinity, with $r$ serving no role.

Suppose $M^{alt}_{[s]}$ is given and has corank $k_s$, and we wish to find the distribution of the corank $k_{s+1}$. In $M^{alt}_{[s+1]}$, take $\ww_0$ to be column $s+1$ with rows $s + 1$ and $2s+2$ removed, and take $\ww_1$ to be column $2s+2$ with the same rows removed. Then, taking
\[b = \left( \frac{dn/p_{s+1}}{p_{s+1}} \right)_+,\]
we get that $M^{alt}_{[s+1]}$ is equivalent to 
\[\left( \begin{array}{lll} M^{alt}_{[s]} & \ww_0 & \ww_1 \\[3pt] \ww_0^{\top} & 0 & b \\ \ww_1^{\top} & b & 0 \end{array}\right).\]
Notice that $b$ is independent from all the other bits in $M^{alt}_{[s+1]}$. If either $\ww_0$ or $\ww_1$ is not in the column space of $M^{alt}_{[s]}$, $b$ has no effect on the rank. However, if both $\ww_0$ and $\ww_1$ are in the column space, then $k_{s+1} - k_s = 2$ for $b = 0$ and $k_{s+1} - k_s = 0$ otherwise. Take $E(\ww)$ to be true if $\ww$ is in the column space of $M^{alt}_{[s]}$ and false otherwise. Then, using the fact that $b$ is independently distributed, we get
\begin{itemize}
\item $k_{s + 1} - k_s = +2$ with probability 
\[0.5 \cdot \mathbb{P}\big(E(\ww_0) \wedge E(\ww_1)\big).\]
\item $k_{s+1} - k_s = 0$ with probability
\[0.5 \cdot \mathbb{P}\big(E(\ww_0)\big) + 0.5 \cdot \mathbb{P}\big(E(\ww_0) \wedge \neg E(\ww_1)\big)\]
\[+\mathbb{P}\big(\neg E(\ww_0) \wedge \big(E(\ww_1) \vee E(\ww_0 + \ww_1)\big)\big). \]
\item $k_{s+1} - k_s = -2$ with probability
\[ \mathbb{P}\big(\neg E(\ww_0) \wedge \neg E(\ww_1) \wedge \neg E(\ww_0+ \ww_1)\big).\]
\end{itemize}
If $\ww_0$ and $\ww_1$ were pulled randomly from the space of $2s + t'$ dimensional vectors, we would easily get that these probabilities are
\begin{itemize}
\item $2^{-2k_s -1}$,
\item $3 \cdot 2^{-k_s} - 5 \cdot 2^{-2k_s -1}$, and
\item $1 - 3 \cdot 2^{-k_s} + 2^{-2k_s + 1}$
\end{itemize}
respectively. In this case, we would have a Markov chain, and with the clear fact that $k_s \equiv t' \mod 2$, we would get the distribution in the proposition.

Sadly, $\ww_0$ and $\ww_1$ do not come from this distribution. However, except in a vanishingly small percentage of cases, this turns out not to effect the distribution. Take $M^{alt}_{[s]}$ to be predetermined. Given this, take $Y_s$ to be the set of possible tuples $(\ww_0, \ww_1)$, and take $H$ to be the space of linear functionals on $2s + t'$ dimensional vectors that vanish on the column space of $M^{alt}_{[s]}$. Then
\[\mathbb{P}\big(E(\ww_0) \wedge E(\ww_1)\big) = \frac{1}{4^{k_s}|Y_s|}\sum_{(\ww_0, \ww_1) \in Y_s}\sum_{h, h' \in H} (-1)^{h(\ww_0) + h'(\ww_1)}.\]
In particular, we see that this probability is $4^{-k_s}$ unless there is a choice of $h, h' \in H$ not both zero so that $h(\ww_0) + h'(\ww_1)$ is zero for all $(\ww_0, \ww_1) \in Y_s$. Doing this for the other transition probabilities, we find that $k_{s+1} - k_s$ assumes the nice transition probabilities given above unless there is such a choice of $h$ and $h'$. But we have the following lemma.

\begin{lem}
\label{lem:swin}
Given $M^{alt}_{[s]}$ with no set $S_t \subseteq [t']$ which corresponds to a square in both $Q_1$ and $Q_2$, take $H$ to be the space of linear functionals on $2s + t'$ variables that vanish on the column space of $M^{alt}_{[s]}$, and take $Y_s$ to be the vector space of possible $(\ww_0, \ww_1)$. Across the $2^{\frac{1}{2}(r + 2c + 4)(r - 1)}$ bit assignments, the proportion of matrices with $h, h' \in H$ not both zero so that $h(\ww_0) + h(\ww_1)$ is zero for $(\ww_0, \ww_1) \in Y_s$ is 
\[\mathcal{O}\left(0.875^{s}\right)\]
 with constant depending on $c$ but not $s$.
\end{lem}

This lemma is very similar to the lemmas four through seven in \cite{Swin08}.
\begin{proof}[Proof of Lemma \ref{lem:swin}]
The element $h$ can be specified by a subset $S$ of $[2s + t']$ via $h(\ww) = \sum_{i \in S} w_i$, and $h'$ can be specified by another subset $S'$. Take $S_0$ to be the set of $i \le s$ so $i \in S$, take $S_1$ to be the set of $i \le s$ so $i + s \in S$, and take $S_2$ to be the set of $i \le t'$ so $2s + i \in S$, similarly defining $S'_0$, $S'_1$, and $S'_2$. We also take $\ww_{0, i}$ to refer to column $i$ of $M^{alt}_{[s]}$, and $\ww_{1, i}$ to refer to column $i + s$ of $M^{alt}_{[s]}$, for $i \le s$.

For $h, h'$ to obey the conditions of the lemma, we need $S_1$ and $S_0'$ to be equal, as we can otherwise find a $\left(\frac{p_i}{p_{s+1}}\right)_+$ in one but not the other, making $h(\ww_0) + h'(\ww_1)$ nonzero for some choice of $\ww_0, \ww_1$. We have five cases, based on the number of distinct nonempty sets among $S_0, S_1, S'_1$.
\begin{description}
\item [No distinct nonempty sets]

In this case, the $h, h'$ correspond to subsets of $[t']$, of which there are a constant number. At the same time, the assumption we made about $S_t$ in the lemma's statement implies that either $h(\ww_{0, i})$ or $h(\ww_{1, i})$ is of the form $\left( \frac{d}{p_i} \right)_+$ for some nonsquare $d | 2D$ if $h$ is nontrivial. These bits are independent for all $i \le s$, so the probability of this case is $\mathcal{O}\left(0.5^s\right)$. If $h$ is trivial, $h'$ is nontrivial, and the same approach works for it.

\item [One distinct nonempty set]

We recall that we assumed that none of $b_1$, $b_2$, and $-b_1b_2$ were squares when we were first defining $M^{alt}$. This is where this comes into play.

For a given $h, h'$, call the distinct nonempty set $T$. Suppose for the moment that an odd number of the $S_0, S_1, S'_1$ are equal to $T$. For any $p_i$, we can choose a $p_{s+1}$ so $p_i/p_{s+1}$ is a square mod $8d$ and so $\left(\frac{p_i}{p_j} \right) = \left( \frac{p_{s+1}}{p_j} \right)$ for all $j \ne i$. In this case, $(\ww_0, \ww_1)$ is very similar to $(\ww_{0, i}, \ww_{1, i})$, with the only potential difference being at positions $i$ and $i+s$ of these vectors. At these positions, we use the definitions of $B_l$ and quadratic reciprocity to say
\[\begin{array}{ll} 
		\left(\frac{b_1}{p_i}\right) & = (\ww_{0, i})_i - (\ww_0)_i \\
		\left(\frac{-1}{p_i}\right) & = \big((\ww_{0, i})_{i+s} + (\ww_{1, i})_i \big) - \big((\ww_{0})_{i+s} + (\ww_{1})_i\big) \\
		\left(\frac{b_2}{p_i} \right)_+ & = (\ww_{1, i})_{i+s} - (\ww_1)_{i+s} \end{array}\]
In the case that only $S_0$ is nonempty, for example, we see that if $i \in T$ and $\left(\frac{b_1}{p_i} \right)_+ = 1$, then $h(\ww_{0, i}) - h'(\ww_{1, i}) \ne 0$, against our assumption. The other cases are similar:
\begin{itemize}
\item If $S_1 = T$ and $S_0 = S_1' = \emptyset$, then $\left( \frac{-1}{p_i} \right)_+ = 0$ for $i \in T$.
\item If $S_1' = T$ and $S_0 = S_1 = \emptyset$, then $\left(\frac{b_2}{p_i}\right)_+ = 0$ for $i \in T$.
\item If $S_0 = S_1 = S_1' = T$, then $\left(\frac{-b_1b_2}{p_i} \right)_+ = 0$ for $i \in T$.
\end{itemize}
At the same time, given $j \in T$, at least one of $h'(\ww_{0, i})$, $h(\ww_{0, i})$, and $h(\ww_{1, i})$ would be effected by a change in $\left(\frac{p_i}{p_j} \right)$ for any $i \ne j$. Together, we have on the order of $|T| + s$ independent bits that need to be properly set for $h, h'$ to lie in $H$. There are $\binom{s}{|T|}$ ways to choose $T$ of size $|T|$, and so we get that the probability of any such $h, h'$ pair is on the order of
\[\mathcal{O}\left(\sum_{|T|} \binom{s}{|T|} 2^{-s - |T|} \right) = \mathcal{O}(0.75^s)\]
by the binomial theorem.

So we now suppose that an even number of $S_0, S_1, S'_1$ are nonempty. Since we already dealt with the all-empty case, this means we need to consider the case where exactly one of these is empty. The method will be clear from just considering the case where $S'_1$ is empty.

We note that, if $p_i/p_j$ is a square mod $8D$, and if $i$ is in $T$ while $j$ is not in $T$, then
\[h'(\ww_{0, i}) + h'(\ww_{0, j}) = \left(\frac{b_1}{p_i} \right).\]
Exactly one half of the classes of $\QQ^{\times}/(\QQ^{\times})^2$ mod $8D$ correspond to primes with $\left(\frac{b_1}{p}\right)_+ = 0$; then either at most one half of the remainder is ever attained by $p_i$ with $i \in T$, or at most one half the remainder is ever attained by $p_j$ with $j$ outside $T$. Then the probability of the $h'(\ww_{0, i})$ equaling zero everywhere is on the order of
\[\mathcal{O}\left(0.75^{\text{min}(|T|, s - |T|)}\right).\]

At the same time, the values of the $\left(\frac{p_i}{p_j}\right)$ are completely independent from to mod $8D$ information, and we get the same $2^{-s}$ as before from $h(\ww_{0, i})$ equaling zero everywhere. The binomial theorem gives that the total probability of there being an $h, h'$ fitting this case is on the order of 
\[\mathcal{O}(0.875^s)\]

\item [Two distinct nonempty sets]
Take $T$ to be the set of indices in an odd number of $S_0$, $S_1$, and $S'_1$. As before, we find that, for $i \in T$, $h(\ww_{0, i}) + h'(\ww_{1, i})$ is determined by $\left(\frac{d}{p_i}\right)$, with $d$ one of $-1$, $b_1$, $b_2$, and $-b_1b_2$. At the same time, via the $\left(\frac{p_i}{p_j}\right)$, we see that two of $h'(\ww_{0, i})$, $h(\ww_{0, i})$, and $h(\ww_{1, i})$ are independently determined. We see that there are on the order of $2^s \binom{s}{|T|}$ choices of $h, h'$ in this case with $T$ of a given size, so the total probability is on then on the order of
\[\mathcal{O} \left(\sum_{|T|} 2^s \binom{s}{|T|}  \cdot 2^{-2s - |T|}\right) = \mathcal{O}(0.75^s).\]

\item [Three distinct but dependent nonempty sets]
By the word dependent in this case, we mean that $S'_1$ is the symmetric difference of $S_0$ and $S_1$, with $S_0$ and $S_1$ distinct and nonempty. In this case, $h(\ww_{0, i}) + h'(\ww_{1, i})$ is determined by
\begin{itemize}
\item $\left(\frac{-b_1}{p_i} \right)$ for $i$ in $S_0 \cap S_1$, 
\item $\left(\frac{b_1b_2}{p_i} \right)$ for $i$ in $S_0 \cap S_1'$, and 
\item $\left(\frac{-b_2}{p_i} \right)$ for $i$ in $S_1 \cap S_1'$.
\end{itemize}

We assumed that at least one of $b_1, b_2$ was not a square times $-1$, so at least two of $-b_1$, $-b_2$, $b_1b_2$ are not squares. Then, taking $T$ to be the smallest of $S_0, S_1, S'_1$, there will be at least $|T|$ independent conditions. These add with the $2s$ bits corresponding to $\left( \frac{p_i}{p_j}\right)$, giving a total probability of $\mathcal{O}(0.75^s)$.

\item [Three independent sets]
Take $T$ to be the set of indices in an odd number of $S_0, S_1, S'_1$. Then the value $h(\ww_{0, i}) + h'(\ww_{1, i})$ is again determined by a bit of information for $i \in T$. At the same time, we get that all three values $h'(\ww_{0, i})$, $h(\ww_{0, i})$, and $h(\ww_{1, i})$ are independent, for $3s + |T|$ total independent bits determining if $h, h'$ are in $H$. The binomial theorem again gives a total probability of $\mathcal{O}(0.75^s)$. This finishes the case, proving the lemma.
\end{description}
\end{proof}

The end of Swinnerton-Dyer's proof, which shows that a Markov process with exponentially decaying weirdness still converges normally, is then enough to prove the proposition.
\end{proof}

\begin{ex}
\label{ex:AAT}
Consider
\[ \left(\begin{array}{ll} 0 & A^{\top} \\ A & 0 \end{array} \right)\]
with $n$ equal to $3$ mod $4$. In this example, $b_1$ and $b_2$ are both equal to $1$, so Lemma \ref{lem:swin} fails. However, it fails in a predictable way, reflecting a Markov chain with different probabilities. We go through this now.

In this example, $\delta = 2$, and we consider the corank $k_s$ of $A_{[s]} = A\big[[s], [s]\big]$ for $s < r - 1 $. In $A_{[s+1]}$, take $\ww_{0}$ to be the first $s$ elements of the final column, take $\ww_1^{\top}$ to be the first $s$ elements of the final row, and take $b$ to be element $s+1$ of column $s+1$. $b$ is still independently distributed. Taking $E(\ww)$ to be true if $\ww$ is in the column space of $A_{[s]}$ and taking $E^{\top}(\ww)$ to be true if $\ww^{\top}$ is in the row space of $A_{[s]}$, we get that the probability that $k_{s+1} - k_s$ is $+1$ is $0.5\mathbb{P}(E(\ww_0) \wedge E^{\top}(\ww_1))$, the probability that it is $-1$ is $\mathbb{P}(\neg E(\ww_0) \wedge \neg E^{\top}(\ww_1))$, and the probability that it is $0$ is what remains.

These probabilities will be nicely calculable unless there is a functional $h$ that is trivial on the column space of $A_{[s]}$ and another functional $h'$ killing the row space so that $h(\ww_0) + h'(\ww_1) = 0$ for all choices of $\ww_0$ and $\ww_1$ with one of $h, h'$ nontrivial. From considering the $\left(\frac{p_i}{p_{s+1}}\right)$, we see that the subset of $[s]$ associated with $h$ and $h'$ must be the same. Call it $S$. As in the second case of the proof of the lemma above, if $i \in S$, then $h, h'$ can only satisfy the condition on the column and row space if $\left(\frac{-1}{p_i}\right) = +1$. The same binomial argument applies as in the lemma.

Then, as $s$ approaches infinity, the probability that $k_{s+1} - k_s$ is $+1$ approaches $2^{-2k_s - 1}$, the probability that it is $-1$ approaches $1 - 2^{-k_s + 1} + 2^{-2k_s}$, and the probability that it is $0$ approaches $2^{-k_s + 1} - 3 \cdot 2^{-2k_s - 1}$. But notice that $k$ models the rank of four torsion in the class group of $\QQ(\sqrt{-n})$. The next part of our analysis, which converts density over bit distributions to natural density, will still apply. Solving the Markov chain we get that, among positive squarefree $n \equiv 3\,(4)$, the proportions with $4$-class rank $k$ is
\[2^{-k^2} \cdot \prod_{i = 1}^k \left(1 - 2^{-i}\right)^{-2} \cdot \prod_{i=1}^{\infty} \left(1 - 2^{-i}\right).\]
A very similar argument works for other imaginary quadratic fields, giving the same distribution.

The first evidence for this distribution was found by Gerth in 1984 \cite{Gert84}, and the distribution was first proved by Fouvry and Kl{\"u}ners in 2007 \cite{Fouv07}. The proof by Fouvry and Kl{\"u}ners was based heavily on Heath-Brown's work in \cite{Heat94}, so this example is another instance of the Kane and Swinnerton-Dyer approach substituting for Heath-Brown's approach.
\end{ex}

\subsection{Converting to natural density}
\label{ssec:kane}
In this section, we use Kane's work to finish the proof of Theorem \ref{thm:ksd_ugly}. We begin by deriving a slight variant of Proposition 9 from his paper. We adopt his notation, taking $S_{N, r, D}$ to denote the set of ordered tuples of primes $(p_1, \dots, p_r)$ whose product is at most $N$ so that none of the $p_i$ divide $2D$. We also denote by $G$ the $\mathbb{F}_2$ vector space $(\Z/8D\Z)^{\times}\big/ \left((\Z/8D\Z)^{\times}\right)^2$. 

\begin{prop}\label{prop:prop9}\cite[Proposition 9]{Kane13}
Take $D$ to be an odd squarefree integer as in Theorem \ref{thm:ksd_ugly}, and choose any $\epsilon > 0$. Then there exists a $K > 0$ depending on $D$ and $\epsilon$ so that, for any choice of $r > 0$ and any $N$ satisfying $(\log \log N)/2 < r < 2 \log \log N$, and for any choice of an alternating $r \times r$ matrix $E$ over $\mathbb{F}_2$ and any choice of a function $F: G^{r} \rightarrow [-1, 1]$,  we have
\begin{equation}
\label{eq:prop9}
\left| \frac{1}{r! \cdot N} \sum_{S_{N, r, D}} F \circ \pi (p_1, \dots, p_r)  \prod_{0 < i < j \le r}\left( \frac{p_i}{p_j} \right)  \right| < K \epsilon^m
\end{equation}
where $m$ is the number of rows of $E$ that are not identically zero and $\pi$ is the natural projection $S_{N, r, D} \rightarrow G^r$.
\end{prop}
\begin{proof}
This is a straightforward variant of Kane's proposition, and it can be proved without dealing with Siegel zeros. Write $Q_{max}(D, N_0, r, m)$ for the supremum of the expression
\[\left| \frac{1}{r! \cdot N} \sum_{S_{N, r, D}} \left( \prod_{0 < i \le r} \chi_i(p_i) \right)F \circ \pi (p_1, \dots, p_r) \prod_{0 < i < j \le r}\left( \frac{p_i}{p_j} \right)  \right| \]
over all $N < N_0$, all $E$ with $m$ nontrivial rows, all $F$ defined mod $8D$, and all choices of quadratic characters $\chi_i$. We are assuming no relation between the nonnegative integers $N_0$ and $r$ here.

Note that $Q_{max}$ always exists, being bounded trivially by $1$ for all $D, N_0, r, m$.
There is a $C > 0$ depending on $D$ so that, for any choice of $A < N_0$, we have
\[ \begin{array}{ll}
Q_{max}(D, N_0, r, m) < C \cdot \bigg( &A^{-1/8}\log^2 N + \frac{\log \log A}{r} Q_{max}(D, N_0, r - 1, m - 1) \\ &+  \left( \frac{\log \log A}{r} \right)^2 Q_{max}(D, N_0, r-2, m-2) \bigg) \end{array}.\]
if $r, m \ge 2$. The argument is very similar to the proof of Lemma 16 in Kane's paper. The only concern is that we need to write 
\[\begin{array}{l}\chi_{r-1}(p_{r-1})\chi_r(p_r)F_{(p_1, \dots, p_{r-2})}(p_{r-1}, p_r)\\[2pt]
 \,\,\,\,\,\, := \chi_{r-1}(p_{r-1})\chi_r(p_r)F \circ \pi(p_1, \dots, p_{r-2}, p_{r-1}, p_r)\end{array}\]
as a finite linear combination of products $a(p_{r-1})b(p_r)$. This is very easy to do if $F_{(p_1, \dots, p_{r-2})}$ is a characteristic function of a single element in $G^2$, and all the other cases are sums of at most $|G|^2$ functions of this form. Since $G$ is finite, the argument of the lemma still works. Taking $A = \log^{16 + c}(N)$ for large enough $c$ is then enough to give the proposition inductively, as in Kane's paper; we are using $Q_{max}(D, N_0, r, 0) \le 1$ as the base case.
\end{proof}

We now sketch how to follow Kane's approach to this slightly new problem. We are in the context of Theorem \ref{thm:ksd_ugly}, so we are considering $M^{alt}(n)$ for $n \in X_{n_0, D}$, with $n = p_1\dots p_r$. $M^{alt}(n)$ defines an alternating form on $2r + t$ dimensional vectors over $\mathbb{F}_2$. Call the space of such vectors $V$. Then we get
\[2^{\text{crnk}(M^{alt}(n))} =2^{-2r - t}\cdot\sum_{\vv, \vv' \in V} (-1)^{\vv^{\top} M^{alt}(n) \vv'}.\]
For $k$ a nonnegative integer, define an alternating form $\Phi^k_{(p_1, \dots, p_r)}$ on $V^{\oplus k}$ by
\[\Phi^k_{(p_1, \dots, p_r)}((\vv_1, \dots, \vv_k), (\vv'_1, \dots, \vv_k)) := \sum_{1 \le i \le k} \vv_i^{\top} M^{alt}(n) \vv_i'.\]
Then we get
\[2^{k \cdot \text{crnk}\, (M^{alt}(n))} = 2^{-k(2r + t)} \cdot \sum_{\bar{\vv}, \bar{\vv}' \in V^{\oplus k}} (-1)^{\Phi^k_{(p_1, \dots, p_r)}(\bar{\vv}, \bar{\vv}')}.\]
In particular, the average size of $2^{k \cdot \text{crnk}\, (M^{alt}(n))}$ over $S_{N, r, D}$ is
\[\frac{2^{-k(2r + t)}}{|S_{N, r, D}|} \sum_{S_{N, r, D}} \sum_{\bar{\vv}, \bar{\vv}' \in V^{\oplus k}} (-1)^{\Phi_{(p_1, \dots, p_r)}^k(\bar{\vv}, \bar{\vv}')},\]
assuming there is at least one such $n$.

Switching the order of summation reveals how Proposition \ref{prop:prop9} is used. At this point, Kane's argument works essentially without modification. Suppose that $\log \log N/2 < r < 2\log \log N$. Then, in this sum, the $(\bar{\vv}, \bar{\vv}')$ that lead to large $m$ in the context of Proposition \ref{prop:prop9} contribute negligibly to the eventual sum as $r$ increases. However, as Kane showed, the number of $(\bar{\vv}, \bar{\vv}')$ with $m$ small is vanishingly small compared to $2^{k(2r + t)}$ as $r$ increases, so these terms are also negligible. Then the only terms that matter are those that have $m = 0$, and hence depend only on the $p_i$ mod $8D$. Proposition 10 of \cite{Kane13} smooths out this information, and we find that, as $r$ increases, the average of $2^{k \cdot \text{crnk}\, (M^{alt}(n))}$ over $S_{N, r, D}$ approaches its average over all bit assignments of the associated Legendre symbols.

Since all but vanishingly few squarefree $n < N$ have between $(\log \log N)/2$ and $2\log\log N$ prime divisors, we have that the average sizes over all squarefree $n < N$ coprime to $2D$ approach the same limit as the average sizes over all bit assignments as $N$ increases. Proposition \ref{prop:sel3} then tells us that these average sizes are consistent with the $\alpha_k$ of Theorem \ref{thm:ksd_ugly}, and the final section of Kane's paper reverse engineers the distribution of coranks from the average values of the $2^{k \cdot \text{crnk}\, (M^{alt}(n))}$. This finishes the argument, proving Theorem \ref{thm:ksd_ugly}.

\bibliography{references}{}
\bibliographystyle{amsplain}

\end{document}